\input ./style/arxiv-general.cfg
\documentclass[aop,MSNbibl,seceqn,dvips]{arximspdf}
\makeatletter
   \@ifpackageloaded{graphicx}{}{\usepackage{graphicx}}
\makeatother
\usepackage{colonequals,subeqn}
\usepackage{url,breakurl}
\doi{10.1214/14-AOP995}% Updated by VTEXPTS2LaTeX.exe, 23.02.2015 09:31
\volume{44}
\issue{2}
\pubyear{2016}
\firstpage{1013}
\lastpage{1052}
\docsubty{FLA}

\makeatletter
\newcommand{\eqref}[1]{(\ref{#1})}
\newtheorem{theorem}{Theorem}[section]
\newtheorem{lemma}[theorem]{Lemma}
\newtheorem{proposition}[theorem]{Proposition}
\newtheorem{corollary}[theorem]{Corollary}

\newproclaim{definition}[theorem]{Definition}
\newproclaim{assumption}[theorem]{Assumption}
\newproclaim{remark}[theorem]{Remark}

\newcommand{\eps}{\varepsilon}
\renewcommand{\P}{\mathbb{P}}
\newcommand{\C}{\mathbb{C}}
\newcommand{\D}{\mathbb{D}}
\newcommand{\E}{\mathbb{E}}
\newcommand{\N}{\mathbb{N}}
\newcommand{\Z}{\mathbb{Z}}
\newcommand{\Q}{\mathbb{Q}}
\newcommand{\R}{\mathbb{R}}
\newcommand{\lp}{\nu}
\newcommand{\gp}{\alpha}
\newcommand{\mgfparam}{\lambda}
\newcommand{\mgfparamtwo}{\eta}
\newcommand{\gffparam}{\sigma}

\newcommand{\CC}{\mathcal{C}}
\newcommand{\CD}{\mathcal{D}}
\newcommand{\CF}{\mathcal{F}}
\newcommand{\CQ}{\mathcal{Q}}
\newcommand{\CS}{\mathcal{S}}
\newcommand{\CU}{\mathcal{U}}

\newcommand{\CH}{\mathcal{H}}

\newcommand{\dist}{\operatorname{dist}}
\newcommand{\diam}{\operatorname{diam}}
\newcommand{\SLE}{\operatorname{SLE}}
\newcommand{\CLE}{\operatorname{CLE}}
\newcommand{\confrad}{\operatorname{CR}}
\newcommand{\inrad}{\operatorname{inrad}}
\newcommand{\essinf}{\operatorname{ess\, inf}}
\newcommand{\esssup}{\operatorname{ess\, sup}}

\newcommand{\Loop}{\mathcal{L}}
\newcommand{\Loopcount}{\mathcal{N}}
\newcommand{\SLoopcount}{\mathcal{S}}
\newcommand{\TLoopcount}{\widetilde{\mathcal{N}}}
\newcommand{\TLoopsum}{\widetilde{\mathcal{S}}}

\newcommand{\given}{ | }

\newcommand{\one}{\mathbf{1}}

\newcommand{\ol}{\overline}
\newcommand{\wh}{\widehat}
\newcommand{\wt}{\widetilde}

\makeatother

\begin{document}
\begin{frontmatter}

%\dochead{}
\title{Extreme nesting in~the~conformal~loop~ensemble}
\runtitle{Extreme nesting in CLE}

\begin{aug}
% Corresponding author: Jason Miller - jpmiller@mit.edu% Updated by
%VTEXPTS2LaTeX.exe, 23.02.2015 09:31
\author[A]{\fnms{Jason}~\snm{Miller}\corref{}\thanksref{T1,m1,m2}\ead[label=e1]{jpmiller@mit.edu}},
\author[A]{\fnms{Samuel S.}~\snm{Watson}\thanksref{T2,m1,m2}} %\ead[label=]{}
\and
\author[B]{\fnms{David B.}~\snm{Wilson}\thanksref{m1}\ead[label=e3]{David.Wilson@microsoft.com}\ead[label=u3]{http://dbwilson.com}} %
\runauthor{J. Miller, S.~S. Watson and D.~B. Wilson}
\thankstext{T1}{Supported in part by Grant DMS-12-04894.}
\thankstext{T2}{Supported in part by an NSF Graduate Research Fellowship, award No. 1122374.}
\affiliation{Microsoft Research\thanksmark{m1} and Massachusetts Institute of Technology\thanksmark{m2}}
%\dedicated{}
\address[A]{J. Miller\\
S.~S. Watson\\
Department of Mathematics\\
Massachusetts Institute of Technology\\
Cambridge, Massachusetts 02142\\
USA\\
\printead{e1}}
\address[B]{D.~B. Wilson\\
Microsoft Research\\
Redmond, WAashington 98052\\
USA\\
\printead{e3}\\
\printead{u3}}
\end{aug}

% HISTORY:
%
\received{\smonth{1} \syear{2014}}% Updated by VTEXPTS2LaTeX.exe,
%23.02.2015 09:31
%
\revised{\smonth{12} \syear{2014}}% Updated by VTEXPTS2LaTeX.exe,
%23.02.2015 09:31

% ABSTRACT
%
\begin{abstract}
The conformal loop ensemble $\CLE_\kappa$ with parameter $8/3 <
\kappa<
8$ is the canonical conformally invariant measure on countably infinite
collections of noncrossing loops in a simply connected domain. Given
$\kappa$ and $\nu$, we compute the almost-sure Hausdorff dimension of the
set of points $z$ for which the number of CLE loops surrounding the disk
of radius $\eps$ centered at $z$ has asymptotic growth $\nu\log
(1/\eps)$
as \mbox{$\eps\to0$}. By extending these results to a setting in which
the loops are given i.i.d. weights, we give a CLE-based treatment of the
extremes of the Gaussian free field.
\end{abstract}

% KEYWORDS
% Pirmas kwd is didziosios raides
%
\begin{keyword}[class=AMS]
\kwd[Primary ]{60J67}
\kwd{60F10}
\kwd[; secondary ]{60D05}
\kwd{37A25}
\end{keyword}
\begin{keyword}
\kwd{SLE}
\kwd{CLE}
\kwd{conformal loop ensemble}
\kwd{Gaussian free field}
\end{keyword}
\end{frontmatter}

%s1 #&#
\section{Introduction}
\label{secpintroduction}

The conformal loop ensemble $\CLE_\kappa$ for $\kappa
\in(8/3,8)$ is the
canonical conformally invariant measure on countably infinite collections
of noncrossing loops in a simply connected domain \label{notpdomD} $D
\subsetneq\C$ \cite{SHE_CLE,SW_CLE}. It is the loop analogue of
$\SLE_\kappa$, the canonical conformally invariant measure on noncrossing
paths. Just as $\SLE_\kappa$ arises as the scaling limit of a single
interface in many two-dimensional discrete models, $\CLE_\kappa$ is a
limiting law for the joint distribution of all of the interfaces.
Figures \ref{figpOn} and \ref{figpFK} show two discrete loop models
believed or known to have $\CLE_\kappa$ as a scaling limit.
Figure~\ref{figpCLE_examples} illustrates these scaling limits
$\CLE_\kappa$ for several values of $\kappa$.

%f1 #&#
\begin{figure}

\includegraphics{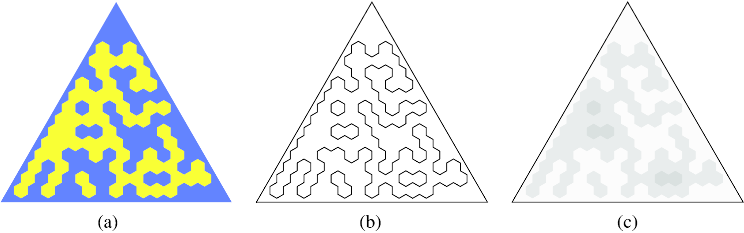}

\caption{Nesting\vspace*{1pt} of loops in the $O(n)$ loop model.
Each $O(n)$ loop configuration has probability proportional to
$x^{\mathrm{total\ length\ of\ loops}} \times n^{\mathrm{\#\ loops}}$.
For a certain critical value of $x$,
the $O(n)$ model for $0\leq n\leq2$ has a ``dilute
phase,'' which is believed to
converge $\CLE_\kappa$ for $8/3<\kappa\leq4$ with $n=-2\cos(4\pi
/\kappa)$.
For $x$ above this critical value, the $O(n)$ loop model
is in a ``dense phase,''
which is believed to converge to $\CLE_\kappa$ for $4\leq\kappa\leq8$,
again with $n=-2\cos(4\pi/\kappa)$. See \cite{KN04} for further background.
\textup{(a)} Site percolation. \textup{(b)} $O(n)$ loop model. Percolation corresponds to $n=1$ and
$x=1$, which is in the dense phase. \textup{(c)} Area shaded by nesting of loops.}
\label{figpOn}
\end{figure}

%\begin{figure}[h!]
%\begin{center}
%\subfigure[]{\includegraphics[width=.32
%\textwidth]{figures/On-spin}}
%\hfill
%\subfigure[]{\includegraphics[width=.32
%\textwidth]{figures/On-loop}}
%\hfill
%\subfigure[]{\includegraphics[width=.32
%\textwidth]{figures/On-nest}}
%\end{center}
%
%\end{figure}

%f2 #&#
\begin{figure}[b]

\includegraphics{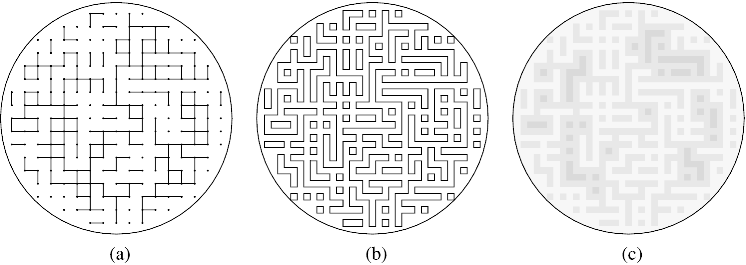}

\caption{Nesting of loops separating critical
Fortuin--Kasteleyn (FK) clusters from dual clusters.
Each FK bond configuration has probability proportional to
$(p/(1-p))^{\mathrm{\#\ edges}} \times q^{\mathrm{\#\ clusters}}$ \cite
{FK}, where
there is believed to be a critical point at $p=1/(1+1/\sqrt{q})$
(proved for $q\geq1$ \cite{BDC}).
For $0\leq q\leq4$, these loops are believed to have the
same large-scale behavior as the $O(n)$ model
loops for $n=\sqrt{q}$ in the dense phase, that is, to
converge to $\CLE_\kappa$ for $4\leq\kappa\leq8$ (see \cite
{RS05,KN04}). \textup{(a)} Critical FK bond configuration. Here $q=2$.
\textup{(b)} Loops separating FK clusters from dual clusters.
\textup{(c)} Area shaded by nesting of loops.}\label{figpFK}
\end{figure}

%\begin{figure}
%\begin{center}
%\subfigure[]{
%\includegraphics[width=.32\textwidth]{figures/FK-bond}}
%\hfill
%\subfigure[]{
%\includegraphics[width=.32\textwidth]{figures/FK-loop}}
%\hfill
%\subfigure[]{\includegraphics[width=.32
%\textwidth]{figures/FK-nest}}
%\end{center}
%
%\end{figure}

%f3 #&#
\begin{figure}

\includegraphics{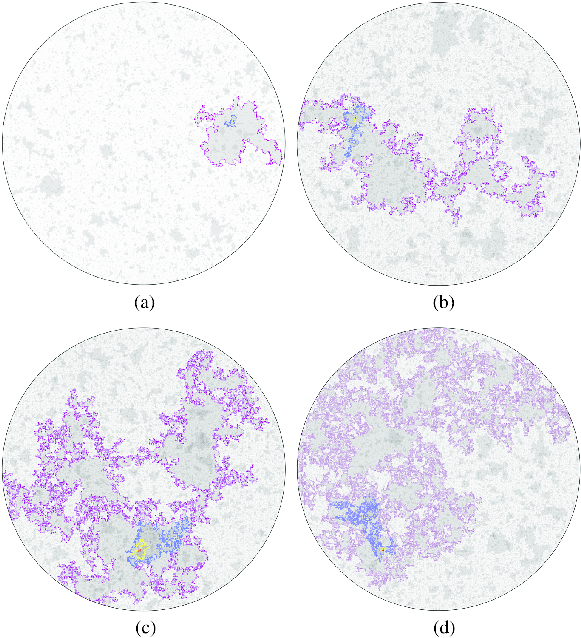}

\caption{Simulations of discrete loop
models which converge to (or are believed to converge to, indicated
with $\star$) $\CLE_\kappa$ in the fine mesh limit.
For each of the $\CLE_\kappa$'s, one particular nested
sequence of loops is outlined. For $\CLE_\kappa$, almost all of the
points in the domain are surrounded by an infinite nested sequence of
loops, though the discrete samples shown here display only a few orders
of nesting. \textup{(a)} $\CLE_3$ (from critical Ising model).
\textup{(b)} $\CLE_4$ (from the FK model with $q=4$) $\star$.
\textup{(c)} $\CLE_{16/3}$ (from the FK model with $q=2$).
\textup{(d)} $\CLE_6$ (from critical bond percolation) $\star$.}\label{figpCLE_examples}
\end{figure}

%\begin{figure}
%\begin{center}
%\subfigure[]{
%\includegraphics[width=0.495
%\textwidth]{figures/nestedloops/nestedloops-3-512-2-lowres}} \hfill
%\subfigure[]{
%\includegraphics[width=0.495
%\textwidth]{figures/nestedloops/nestedloops-4-256-2-lowres}}\\[24pt]
%\subfigure[]{
%\includegraphics[width=0.495
%\textwidth]{figures/nestedloops/nestedloops-163-256-2-lowres}} \hfill
%\subfigure[]{
%\includegraphics[width=0.495
%\textwidth]{figures/nestedloops/nestedloops-6-512-square-lowres}}
%\end{center}
%
%\end{figure}

%s1.1 #&#
\subsection{Overview of main results}

Fix a simply connected domain $D \subsetneq\C$ and let $\Gamma$
\label{notpCLE} be a $\CLE_\kappa$ in $D$. For each point $z \in D$ and
$\eps> 0$, we let $\Loopcount_z(\eps)$ \label{notpLoopcount} be the
number of loops of $\Gamma$ which surround $B(z,\eps)$, the ball of radius
$\eps$ centered at $z$. We study the behavior of the \textit
{extremes} of
$\Loopcount_z(\eps)$ as $\eps\to0$, that is, points where
$\Loopcount_z(\eps)$ grows unusually quickly or slowly
(Theorem~\ref{thmpmain}). We also analyze a more general setting in which
each of the loops is assigned an i.i.d. weight sampled from a given law
$\mu$. \label{notpmu} This in turn is connected with the extremes of the
continuum Gaussian free field (GFF) \cite{HMP} when $\kappa=4$ and
$\mu(\{-\gffparam\}) = \mu(\{\gffparam\}) = \frac{1}{2}$ for a particular
value of $\gffparam> 0$ (Theorems~\ref{thmpgff_maximum} and~\ref{thmpgff_maximum_support}).

%s1.2 #&#
\subsection{Extremes}

Fix $\alpha\geq0$. The Hausdorff $\alpha$-measure $\CH_\alpha$ of
a set
$E\subset\C$ is defined to be
\[
\CH_\alpha(E) = \lim_{\delta\to0} \biggl(\inf \biggl\{ \sum
_{i} \bigl( \diam(F_i)
\bigr)^\alpha \dvtx \bigcup_i
F_i \supseteq E, \diam (F_i) < \delta \biggr\} \biggr),
\]
where the infimum is over all countable collections $\{F_i\}$ of sets. The
Hausdorff dimension of $E$ is defined to be
\[
\dim_\CH(E)\colonequals\inf\bigl\{\alpha\geq0 \dvtx
\CH_\alpha(E)=0\bigr\}.
\]
 For each $z \in D$ and $\eps> 0$, let
%
%e1.1 #&#
\begin{equation}
\label{eqpTloopcount} \TLoopcount_z(\eps)\colonequals\frac{\Loopcount_z(\eps)}{\log
(1/\eps)}.
\end{equation}
For $\lp\geq0$, we define
%
%e1.2 #&#
\begin{equation}
\label{notpPhi_lp} \Phi_\lp(\CLE_\kappa) \colonequals
\Phi_\lp(\Gamma) \colonequals \Bigl\{ z \in D \dvtx \lim
_{\eps\to0} \TLoopcount_z(\eps) = \lp \Bigr\}.
\end{equation}

Our first result gives the almost-sure Hausdorff dimension of
$\Phi_\lp(\CLE_\kappa)$. The dimension is given in terms of the
distribution of the conformal radius of the connected component of the
outermost loop surrounding the origin in a $\CLE_\kappa$ in the unit
disk. More precisely, the \textit{conformal radius} $\confrad(z,U)$
of a
simply connected proper domain $U\subset\C$ with respect to a point
$z\in
U$ is defined to be $|\varphi'(0)|$ where $\varphi\colon\D\to U$ is a
conformal map which sends 0 to $z$. For each $z \in D$, let
\label{notploop} $\Loop_z^k$ be the $k$th largest loop of $\Gamma$ which
surrounds $z$, and let $U_z^k$ be the connected component of the open set
$D\setminus\Loop_z^k$ which contains $z$. Take $D=\D$ and let
$T=-\log(\confrad(0,U_0^1))$. The $\log$ moment generating function
of $T$
was computed in \cite{SSW} and is given by
%
%e1.3 #&#
\begin{equation}
\label{eqnpLambdaR} \Lambda_\kappa(\mgfparam) \colonequals\log\E
\bigl[e^{\mgfparam
T} \bigr] = \log \biggl(  \frac{-\cos(4\pi/\kappa)}{\cos (\pi
\sqrt{ (1-{4}/{\kappa} )^2+
{8\mgfparam}/{\kappa}} )}
\biggr) ,
\end{equation}
for $-\infty< \mgfparam< 1 - \frac{2}{\kappa} - \frac{3\kappa
}{32}$. The
almost-sure value of $\dim_\CH\Phi_\lp(\Gamma)$ is given in terms
of the
\textit{Fenchel--Legendre transform} $\Lambda_\kappa^\star
\dvtx \mathbb{R} \to
[0,\infty]$ of $\Lambda_\kappa$, which is defined by \label{notpfenchel}
\[
\Lambda_\kappa^\star(x)\colonequals\sup_{\mgfparam\in\R
}
\bigl(\mgfparam x - \Lambda_\kappa(\mgfparam)\bigr).
\]
We also define
%
%e1.4 #&#
\begin{equation}
\label{eqnpg(nu)} \gamma_\kappa(\lp) = \cases{
\lp\Lambda_\kappa^\star(1/\lp), &\quad $\mbox {if } \lp> 0$,
\vspace*{2pt}
\cr
\displaystyle 1 - \frac{2}{\kappa} - \frac{3\kappa}{32}, &\quad $\mbox{if }\lp=0
$.}
\end{equation}

See Corollary~\ref{corphow_many_summands} for discussion of the
formula in
\eqref{eqnpg(nu)}.

%th1.1 #&#
\begin{theorem}
\label{thmpmain}
Let $\kappa\in(8/3,8)$, and let $\lp_{\mathrm{max}}$ be the unique
value of
$\lp\geq0$ such that $\gamma_\kappa(\nu) = 2$. If $0\leq\lp\leq
\lp_{\mathrm{max}}$,
then almost surely
%
%e1.5 #&#
\begin{equation}
\label{eqnpmain_dim} \dim_\CH\Phi_\lp(\CLE_\kappa) =
2 - \gamma_\kappa(\lp)
\end{equation}
and $\Phi_\lp(\CLE_\kappa)$ is dense in $D$. If $\lp_{\mathrm
{max}}< \lp$, then
$\Phi_\lp(\CLE_\kappa)$ is almost surely empty. (See Figures \ref{figpdimplot} and \ref{figpavgmax}.)

Moreover, if $\Gamma$ is a $\CLE_\kappa$ in $D$, $\varphi\colon D
\to\acute{D}$ is a conformal transformation, $\acute{\Gamma}
\colonequals
\varphi(\Gamma)$, and $\Phi_\lp(\acute{\Gamma})$ is defined to be the
corresponding set of extremes of $\acute{\Gamma}$, then
$\Phi_\lp(\acute{\Gamma}) = \varphi(\Phi_\lp(\Gamma))$ almost surely.
\end{theorem}

%f4 #&#
\begin{figure}

\includegraphics{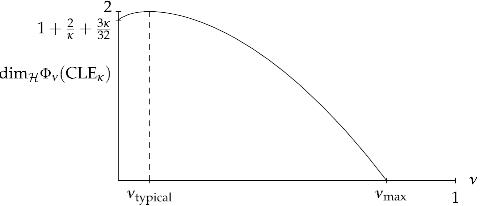}

\caption{Suppose that $D \subsetneq\C$ is a simply
connected domain and let $\Gamma$ be a $\CLE_\kappa$ in $D$. For
$\kappa
\in(8/3,8)$ and $\lp\geq0$, we let $\Phi_\lp(\Gamma)$ be the set of
points $z$ for which the number of loops $\Loopcount_z(\eps)$ of
$\Gamma$
surrounding $B(z,\eps)$ is $(\lp+o(1)) \log(1/\eps)$ as $\eps\to0$.
The plot above shows how the almost-sure Hausdorff dimension of
$\Phi_\lp(\CLE_\kappa)$ established in Theorem \protect\ref
{thmpmain} depends
on $\lp$ (the figure is for $\kappa=6$, but the behavior is similar for
other values of $\kappa$). The value $1+\frac{2}{\kappa} +
\frac{3\kappa}{32}=\dim_\CH\Phi_\lp(\CLE_\kappa)$ is the almost-sure
Hausdorff dimension of the $\CLE_\kappa$ gasket \cite{SSW,NW,MSW}, which
is the set of points in $D$ which are not surrounded by any loop of
$\Gamma$.} \label{figpdimplot}
\end{figure}

%f5 #&#
\begin{figure}[b]

\includegraphics{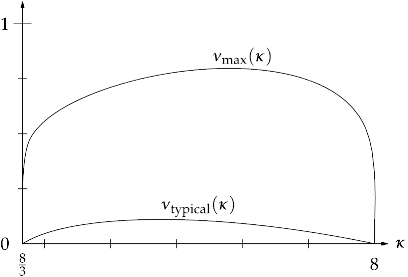}

\caption{The typical nesting and maximal nesting constants ($\lp
_{\mathrm{typical}}$ and
$\lp_{\operatorname{max}}$) plotted versus $\kappa$. For example,
when $\kappa=6$, Lebesgue almost
all points are surrounded by $(0.091888149\ldots+o(1)) \log(1/\eps)$
loops with inradius at least $\eps$, while some points are surrounded by
as many as $(0.79577041\ldots+o(1)) \log(1/\eps)$ loops.}
\label{figpavgmax}
\end{figure}

We also show in Theorem~\ref{thmpuncountable} that
$\Phi_{\nu_{\max}}(\Gamma)$ is almost surely uncountably infinite
for all
$\kappa\in(8/3,8)$. This contrasts with the critical case for thick
points of the Gaussian free field: it has only been proved that the set
of critical thick points is infinite (not necessarily uncountably
infinite); see Theorem~1.1 of \cite{HMP}.

See Figure~\ref{figpdimplot} for a plot of the Hausdorff dimension of
$\Phi_\lp(\CLE_6)$ as a function of $\nu$. The discrete analog of
Theorem~\ref{thmpmain} would be to give the growth exponent of the
set of
points which are surrounded by unusually few or many loops for a given
model as the size of the mesh tends to zero. Theorem~\ref{thmpmain} gives
predictions for these exponents. Since $\CLE_6$ is the scaling limit of
the interfaces of critical percolation on the triangular lattice
\cite{S05,CN06,CN07}, Theorem~\ref{thmpmain} predicts that the typical
point in critical percolation is surrounded by $(0.09189\ldots+o(1))
\log(1/\eps)$ loops as $\eps\to0$, where $\eps> 0$ is the lattice
spacing.

We give a brief explanation of the proof for the case $\lp= 1/ \E T$: by
the renewal property of $\CLE_\kappa$,
the random variables $\log\confrad(z,U_z^k) - \log\confrad(z,U_z^{k+1})$
are i.i.d. and equal in distribution to $T$. It follows from the law of
large numbers (and basic distortion estimates for conformal maps) that, for
$z\in D$ fixed, $\TLoopcount_z(\eps)\to1/\E T \mbox{ as }\eps\to0$,
almost surely. By the Fubini--Tonelli theorem, we conclude that the
expected Lebesgue\vadjust{\goodbreak} measure of the set of points for which
$\TLoopcount_z(\eps) \nrightarrow1/\E T$ is 0. It follows that almost
surely, there is a full-measure set of points $z$ for which
$\TLoopcount_z(\eps)\to1/\E T$. In other words, $\lp=\lp_{\mathrm
{typical}}
\colonequals
1/\E T$ corresponds to typical behavior, while points in
$\Phi_\lp(\CLE_\kappa)$ for $\lp\neq1/\E T$ have exceptional loop-count
growth.

The idea to prove Theorem~\ref{thmpmain} for other values of $\lp$ is
to use a multiscale refinement of the second moment method
\cite{HMP,DPRZ}. The main challenge in applying the second moment
method to obtain the lower bound of the dimension of the set
$\Phi_\lp(\CLE_\kappa)$ in Theorem~\ref{thmpmain} is to deal with the
complicated geometry of $\CLE$ loops. In particular, for any pair of
points $z,w \in D$ and $\eps>0$, there is a positive probability that
single loop will come within distance $\eps$ of both $z$ and $w$. To
circumvent this difficulty, we restrict our attention to a special
class of points $z \in\Phi_\lp(\CLE_\kappa)$ in which we have precise
control of the geometry of the loops which surround $z$ at every
length scale.

The $\CLE$ gasket is defined to be the set of points $z \in D$ which are
not surrounded by any loop of $\Gamma$. Equivalently, the gasket is the
closure of the union of the set of outermost loops of $\Gamma$. Its
expectation dimension, the growth exponent of the expected minimum number
of balls of radius $\eps> 0$ necessary to cover the gasket as $\eps
\to
0$, is given by $1+\frac{2}{\kappa} + \frac{3\kappa}{32}$ \cite
{SSW}. It
is proved in \cite{NW} using Brownian loop soups that the almost-sure
Hausdorff dimension of the gasket when $\kappa\in(8/3,4]$ is
$1+\frac{2}{\kappa} + \frac{3\kappa}{32}$, and it is shown in
\cite{MSW}
that this result holds for $\kappa\in(4,8)$ as well. We show in
Proposition~\ref{propplimit_at_zero} that the limit as $\lp\to0$ of
$\dim_\CH\Phi_\lp(\Gamma)$ is $1+\frac{2}{\kappa} + \frac
{3\kappa}{32}$
(equivalently, $\gamma_\kappa$ is right continuous at $0$).
Consequently, from the
perspective of Hausdorff dimension, there is no nontrivial intermediate
scale of loop count growth which lies between logarithmic growth and the
gasket.

Theorem~\ref{thmpmain} is a special case of a more general result, stated
as Theorem~\ref{thmparbitrary_distribution} in
Section~\ref{secpweighted_loops}, in which we associate with each loop
$\Loop$ of $\Gamma$ an i.i.d. weight $\xi_\Loop$ distributed
according to
some probability measure $\mu$. For each $\gp> 0$, we give the
almost-sure Hausdorff dimension of the set
\[
\Phi^\mu_\gp(\Gamma) \colonequals \Bigl\{z\in D \dvtx
\lim_{\eps\to
0^+} \TLoopsum_\eps(z) = \gp \Bigr\}
\]
of extremes of the normalized weighted loop counts
a graph
%
%e1.6 #&#
\begin{equation}
\label{eqnpnormalized_loop_count} \TLoopsum_z(\eps) = \frac{1}{\log(1/\eps)}
\SLoopcount_z(\eps) \qquad\mbox{where } \SLoopcount_z(\eps) =
\sum_{\Loop\in\Gamma_z(\eps)} \xi_\Loop,
\end{equation}
and $\Gamma_z(\eps)$ is the
set of loops of $\Gamma$ which surround $B(z,\eps)$. This dimension is
given in terms of $\Lambda_\kappa^\star$ and the Fenchel--Legendre transform
$\Lambda_\mu^\star$ of $\mu$. Although the dimension for general weight
measures $\mu$ and $\kappa\in(8/3,8)$ is given by a complicated
optimization problem, when $\kappa=4$ and $\mu$ is a signed Bernoulli
distribution, this dimension takes a particularly nice form. We state
this result as our second theorem.

%th1.2 #&#
\begin{theorem}
\label{thmpgff_maximum}
Fix $\gffparam> 0$, and define $\mu_{\rm B} (\{\gffparam\})=\mu
_{\rm
B}(\{-\gffparam\}) = \frac{1}{2}$. In the special case $\kappa=4$ and
$\mu= \mu_{\rm B}$, almost surely
%
%e1.7 #&#
\begin{equation}
\label{eqnpgff_dim} \dim_\CH\Phi^{\mu_B}_\gp(\Gamma) =
\max \biggl(0,2-\frac{\pi
^2}{2\gffparam^2} \gp^2 \biggr).
\end{equation}
\end{theorem}

This case has a special interpretation which explains the formula
\eqref{eqnpgff_dim} for the dimension. It is proved in \cite{MS_CLE} that
for $\gffparam=\sqrt{\pi/2}$,
the random height field $\SLoopcount_z(\eps)$ converges in the space of
distributions as $\eps\to0$ to a two-dimensional Gaussian free field $h$,
and the loops $\Gamma$ can be thought of as the level sets of $h$. Since $h$
is distribution-valued, $h$ does not have a well-defined value at any
given point, nor does $h$ have
level sets strictly speaking,
but there is a way to make this precise.
This GFF interpretation suggests a correspondence
between the extremes of $\SLoopcount_z(\eps)$ and the extremes of $h$.
The extreme values of $h$ (also called
\textit{thick points}) can be defined by considering the average
$h_\eps(z)$ of $h$ on $\partial B(z,\eps)$ and defining $T(\gp)$ to
be the
set of points $z$ for which $h_\eps(z)$ grows like $\gp\log(1/\eps
)$ as
$\eps\to0$. Thick points were introduced by Kahane in the context of
Gaussian multiplicative chaos (see
\cite{rhodes-vargas}, Section~4, for further background). It is shown in
\cite{HMP} that $\dim_\CH T(\gp) = 2-\pi\gp^2$, which equals $\dim
_\CH
\Phi^{\mu_{\rm B}}$ when $\sigma=\sqrt{\pi/2}$ and $\kappa= 4$. The
following theorem relates exceptional loop count growth with the extremes
of the GFF. Loosely speaking, it says that for each $\alpha$ there is a
unique value of $\nu$ for which ``most'' of the $\alpha$-thick points have
loop counts $\TLoopcount\approx\nu$.

%f6 #&#
\begin{figure}

\includegraphics{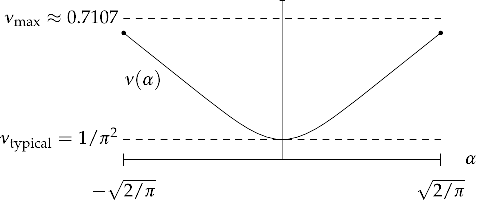}

\caption{A graph of $\lp(\gp)$ versus $\gp$,
which gives the typical loop growth $\lp\log(1/\eps)$ corresponding to
each point with signed loop growth $\gp\log(1/\eps)$, for $\gp\in
 [-\sqrt{2/\pi},\sqrt{2/\pi} ]$. Also shown is the value
$\lp_{\mathrm{max}}$ beyond which there are no points having growth
$\lp
\log(1/\eps)$.}\label{GFF_loop_profile}
\end{figure}

%th1.3 #&#
\begin{theorem}
\label{thmpgff_maximum_support}
Let $\kappa=4$ and $\mu_{\rm B}(\{\sqrt{\pi/2}\})=\mu_{\rm B}(\{
-\sqrt{\pi/2}\}) = \frac{1}{2}$.
For every $\gp\in[-\sqrt{2/\pi},\sqrt{2/\pi}]$, there exists a
unique $\lp= \lp(\gp) \geq0$ such
that the Hausdorff dimension of the set of points with $\TLoopsum
_z(\eps)
\to\gp$ as $\eps\to0$ is equal to the Hausdorff dimension of the set
of points with $\TLoopsum_z(\eps)\to\gp$ and $\TLoopcount_z(\eps
)\to
\lp$ as $\eps\to0$. Moreover,
\[
\lp(\gp) = \frac{\gp}{\sqrt{\pi/2}} \coth \biggl(\frac{\pi^2
\gp}{\sqrt{\pi/2}} \biggr),
\]
see Figure \ref{GFF_loop_profile}.
\end{theorem}

\subsection*{Outline}

We review large deviation estimates and give some basic overshoot
estimates for random walks in Section~\ref{secppreliminaries}. In
Section~\ref{secpCLE}, we give large deviation estimates on the
nesting of CLE loops, and show the CLE loops are well behaved in
certain senses that are useful when we prove the Hausdorff
dimension in Theorem~\ref{thmpmain}. It suffices to describe the
CLE nesting behavior at single points to give the upper bound
(Section~\ref{secpupperbound}). For the lower bound
(Section~\ref{secplowerbound}), we follow the strategy of studying
a subset of special points that have full dimension and are only
weakly correlated. We give a Hausdorff dimension lower bound proposition
that is cleaner than ones that have appeared earlier, and which allowed
for simplifications in the CLE calculations. In this section,
we also show that the points of maximal nesting are equinumerous with
$\R$.
In Section~\ref{secpweighted_loops}, we\vadjust{\goodbreak}
explain the proof of Theorem~\ref{thmparbitrary_distribution}, the
extension of Theorem~\ref{thmpmain} to the setting of weighted $\CLE$
loops. We also deduce Theorems \ref{thmpgff_maximum} and
\ref{thmpgff_maximum_support} as corollaries of this result.\vadjust{\goodbreak}

%s2 #&#
\section{Preliminaries}
\label{secppreliminaries}

In Section~\ref{subsecpFL}, we review some
facts from large deviations, and then in Section~\ref{subsecpgeneral} we
collect several estimates for random walks.

%s2.1 #&#
\subsection{Large deviations}
\label{subsecpFL}

We review some basic results from the theory of large deviations,
including the Fenchel--Legendre transform and\break 
\hyperref[thmpcramer]{Cram\'er's
theorem}. Let $\mu$ be a probability measure on $\R$.
The logarithmic moment generating function, also known as the cumulant
generating function, \label{notpfenchel_two} of $\mu$ is defined by
\[
\Lambda(\mgfparam) = \Lambda_\mu(\mgfparam) = \log\E
\bigl[e^{\mgfparam X} \bigr] ,
\]
where $X$ is a random variable with law $\mu$. The Fenchel--Legendre
transform $\Lambda^\star\dvtx \mathbb{R} \to[0,\infty]$ of $\Lambda$
is given
by \cite{DZ}, Section~2.2
\[
\Lambda^\star(x)\colonequals\sup_{\mgfparam\in\R}\bigl(
\mgfparam x - \Lambda(\mgfparam)\bigr).
\]

We now recall \hyperref[thmpcramer]{Cram\'er's
theorem} in $\R$, as stated in \cite{DZ},
Theorem~2.2.3.

%th2.1 #&#
\begin{theorem}[(Cram\'er's theorem)] \label{thmpcramer} Let $X$ be a
real-valued random variable and let $\Lambda$ be the logarithmic moment
generating function of the distribution of $X$. Let $S_n=\sum_{i=1}^n
X_i$ be a sum of i.i.d. copies of $X$. For every closed set $F \subset
\R$ and open set $G\subset\R$, we have
\begin{eqnarray*}
\limsup_{n\to\infty} \frac{1}{n} \log\P \biggl[
\frac{1}{n} S_n \in F \biggr] &\leq& -\inf_{y\in F}
\Lambda^\star(y)\quad\mbox{and}
\\
\liminf_{n\to\infty} \frac{1}{n} \log\P \biggl[
\frac{1}{n} S_n \in G \biggr] &\geq& -\inf_{y\in G}
\Lambda^\star(y).
\end{eqnarray*}
Moreover,
%
%e2.1 #&#
\begin{equation}
\label{eqnpcramer_strong} \P \biggl[\frac{1}{n} S_n \in F \biggr] \leq2
\exp \Bigl(-n\inf_{y\in F} \Lambda^\star(y) \Bigr).
\end{equation}
\end{theorem}

Following \cite{DZ}, Section~2.2.1, we let $\CD_\Lambda\colonequals
\{\mgfparam\dvtx  \Lambda(\mgfparam) < \infty\}$ and $\CD_{\Lambda
^\star} = \{ x
\dvtx \break  \Lambda^\star(x) < \infty\}$ be the sets where $\Lambda$ and
$\Lambda^\star$
are finite, respectively, and let $\CF_\Lambda= \{\Lambda'(\mgfparam)
\dvtx \mgfparam\in\CD^\circ_\Lambda\}$, where $A^\circ$ denotes the
interior of a set $A\subset\R$. The following proposition summarizes some
basic properties of $\Lambda$ and $\Lambda^\star$.

%pr2.2 #&#
\begin{proposition}
\label{propptransform_properties}
Suppose that $\mu$ is a probability measure on $\R$, let $\Lambda$
be its log moment generating function, and assume that $\CD_\Lambda
\neq\{0\}$. Let $a$ and $b$ denote the essential infimum and
supremum of a $\mu$-distributed random variable $X$ (with $a=-\infty$
and/or $b=\infty$ allowed). Then $\Lambda$ and its
Fenchel--Legendre transform $\Lambda^\star$ have the following properties:
\begin{longlist}[(viii)]
\item[(i)]$\Lambda$ and $\Lambda^\star$ are convex; %\label{itempconvex}
\item[(ii)]$\Lambda^\star$ is nonnegative;\vadjust{\goodbreak}
\item[(iii)]$\CF_\Lambda\subset\CD_{\Lambda^\star}$;
\item[(iv)]$\Lambda$ is smooth on $\CD_\Lambda^\circ$ and
$\Lambda^\star$ is smooth on $\CF_\Lambda^\circ$; %\label{itempsmooth}
\item[(v)] If $\CD_\Lambda= \R$, then
$\CF_\Lambda^\circ=(a,b)$; %\label{itempinterior}
\item[(vi)] If $(-\infty,0]\subset\CD_\Lambda$, then
$(a,a+\delta)\subset\CF_\Lambda^\circ$ for some $\delta>0$; %\label{itempinteriora}
\item[(vii)] If $[0,\infty)\subset\CD_\Lambda$, then
$(b-\delta,b)\subset\CF_\Lambda^\circ$ for some
$\delta>0$; %\label{itempinteriorb}
\item[(viii)]$\Lambda^\star$ is continuously differentiable on $(a,b)$;
%\label{itempdiff}
%
\item[(ix)] If $-\infty< a$, then $(\Lambda^\star)'(x) \to-\infty$ as
$x\downarrow a$;
%\label{itempboundary_derivative_a}
%
\item[(x)] If $b <\infty$, then $(\Lambda^\star)'(x) \to+\infty$ as $x
\uparrow b$.
%\label{itempboundary_derivative_b}
\end{longlist}
\end{proposition}
\begin{pf}
For (i)--(iv), we refer the reader to
\cite{DZ}, Section~2.2.1.

To prove (v), note that
\[
\Lambda'(\lambda) = \frac{\E[ X e^{\lambda X}]}{\E[ e^{\lambda X}]}.
\]
Therefore,
\[
a= \frac{\E[a e^{\lambda X}]}{\E[e^{\lambda X}]} \leq\Lambda '(\lambda) \leq
\frac{\E[b e^{\lambda X}]}{\E[e^{\lambda X}]} = b.
\]
Thus, $\CF_\Lambda\subseteq[a,b]$, which gives $\CF_\Lambda^\circ
\subseteq(a,b)$.

This leaves us to prove the reverse inclusion.
Suppose $c\in(a,b)$, and let $Y=X-c$.

Then
\[
\Lambda'(\lambda) =\frac{\E[X e^{\lambda X}]}{\E[e^{\lambda X}]} =c + \frac{\E[Y e^{\lambda Y}]}{\E[e^{\lambda Y}]} = c
+ \frac{\E[\one_{\{Y \geq0\}} Y e^{\lambda Y}] + \E[\one_{\{Y
< 0\}} Y e^{\lambda Y}]}{
\E[e^{\lambda Y}]}.
\]
Since $\CD_\Lambda=\R$, the tails of $X$ and $Y$ decay rapidly
enough for each
of the above expected values to be finite.
Since $\P[Y>0]>0$, the first term in the numerator
diverges as $\lambda\to\infty$, while the second term decreases
monotonically in absolute value.
So for sufficiently large $\lambda$, we have $\Lambda'(\lambda)>c$.
Similarly, for sufficiently large negative $\lambda$ we have
$\Lambda'(\lambda)\leq c$. Since $\Lambda$ is smooth, $\Lambda'$ is
continuous, so $c\in\CF_\Lambda$. The proofs of (vi) and
(vii) are analogous.

To prove (viii), note that $\mathcal{F}_\Lambda^\circ=
(\tilde{a},\tilde{b})$ for some $a\leq\tilde{a} < \tilde{b}\leq
b$. By
(iv), $\Lambda^\star$ is smooth on
$(\tilde{a},\tilde{b})$. Therefore, it suffices to consider the possibility
that $a<\tilde{a}$ or $\tilde{b}<b$. Suppose first that $\tilde
{b}<b$. By
the proof of (v), $\tilde{b}<b$ implies that
$\CD_\Lambda= (\lambda_1,\lambda_2)$ for some
$\lambda_2<\infty$. Furthermore, observe that $\Lambda'(\lambda)
\to
\tilde{b}$ as $\lambda\nearrow\lambda_2$. It follows that
$\Lambda(\lambda_2)<\infty$, and by convexity of $\Lambda$ we have
for all
$\tilde{b} \leq x < b$,
\[
\Lambda^\star(x) = \sup_{\lambda} \bigl[x\lambda-
\Lambda(\lambda)\bigr] = x\lambda_2 - \Lambda(\lambda_2).
\]
In other words, $\Lambda^\star$ is smooth on $(\tilde{a},\tilde
{b})$ and is
affine on $(\tilde{b},b)$ with slope matching the left-hand derivative at
$\tilde{b}$. Similarly, if $a<\tilde{a}$, then $\Lambda^\star$ is
affine on
$(a,\tilde{a})$ with slope matching the right-hand derivative of
$\Lambda^\star$ at $\tilde{a}$. Therefore, $\Lambda^\star$ is continuously
differentiable on $(a,b)$.

To prove (ix), we note that since $X$
is bounded below, $D_\Lambda^\circ= (-\infty,\xi)$ for some $0 \leq
\xi
\leq+\infty$. Moreover, there exists $\eps> 0$ so that $(a,a+\eps)
\subset\mathcal{F}_\Lambda^\circ$, by essentially the same argument we
used to prove (v) above. Let
$\wh{\mathcal{D}}=\{\lambda\dvtx  \Lambda'(\lambda) \in(a,a+\eps)\}$,
and note that the left endpoint of $\wh{\mathcal{D}}$ is
$-\infty$. Since $\Lambda'$ is smooth and strictly increasing on
$\wh{\mathcal{D}}$ (see \cite{DZ}, Exercise 2.2.24), there exists a
monotone bijective function $\lambda\dvtx  (a,a+\eps) \to\wh{\mathcal{D}}$
for which $\Lambda'(\lambda(x)) = x$.
In the definition of $\Lambda^\star(x)$, the supremum is achieved
at $\lambda=\lambda(x)$. Differentiating, we obtain
%
%e2.2 #&#
%e2.3 #&#
%e2.4 #&#
\begin{eqnarray}
\label{eqpLambda-star-prime} \bigl(\Lambda^\star\bigr)'(x) &=&
\frac{d}{dx} \bigl[x \lambda(x) - \Lambda \bigl(\lambda(x)\bigr) \bigr]\nonumber
\\
&= &\lambda(x) + x \lambda'(x) - \lambda'(x)
\Lambda'\bigl(\lambda(x)\bigr)
\\
&=& \lambda(x).\nonumber
\end{eqnarray}
Since the monotonicity of $\lambda$ implies that $\lambda(x) \to
-\infty$
as $x \to a$, this completes the proof.

The proof of (x) is similar.
\end{pf}

We also have the following adaptation of \hyperref[thmpcramer]{Cram\'er's
theorem} for which the
number of i.i.d. summands is not fixed.

%co2.3 #&#
\begin{corollary} \label{corphow_many_summands}
Let $X$ be a positive real-valued random variable with exponential tails
(i.e., $\E[e^{\lambda_0 X}]<\infty$ for some $\lambda_0>0$), and let
$\Lambda(\lambda)=\break \log\E[e^{\lambda X}]$.
Let $S_n=\sum_{i=1}^n X_i$ be a sum
of i.i.d. copies of $X$, and let $N_r=\min\{n\dvtx S_n\geq r\}$. If
$0<\nu_1<\nu_2$, then
%
%e2.5 #&#
\begin{equation}
\label{eqnphow_many_summands} \lim_{r\to\infty} \frac{1}{r}\log\P[
\nu_1 r \leq N_{r} \leq\nu _2 r] = -\inf
_{\nu\in[\nu_1,\nu_2]}\nu\Lambda^\star(1/\nu).
\end{equation}
\end{corollary}

This is the origin of the expression $\nu\Lambda^\star(1/\nu)$ in
\eqref{eqnpg(nu)}.
\begin{pf}
Note that in the formula
\[
\Lambda^\star(x)= \sup_\lambda\bigl[\lambda x -
\Lambda(\lambda)\bigr] ,
\]
the bracketed expression is $0$ when $\lambda=0$. Because $X$ has
exponential tails, $\Lambda'(0)=\E[X]$ exists. If $x<\E[X]$, then for
some sufficiently small negative $\lambda$, the bracketed expression is
positive, so $\Lambda^\star(x)>0$. Likewise, if $x>\E[X]$ then
$\Lambda^\star(x)>0$. We also have $\Lambda^\star(\E[X])=0$.

Let $a=\essinf X \in[0,\infty)$ be the essential infimum of $X$ and
$b=\esssup X\in(0,\infty]$ be the essential supremum of $X$.

Because $\Lambda^\star$ is convex on $[a,b]$, by Lemma~\ref{lempconvex}
proved below, $\nu\Lambda^\star(1/\nu)$ is convex on $[1/b,1/a]$. The
expression $\nu\Lambda^\star(1/\nu)$ is $0$ when $\nu=1/\E[X]$
and is
positive elsewhere on $[1/b,1/a]$, so it is strictly decreasing for
$\nu\leq1/\E[X]$ and strictly increasing for $\nu\geq1/\E[X]$.

There are three possible cases for the relative order of $1/\E[X]$,
$\nu_1$, and $\nu_2$. For example, suppose $\nu_1<\nu_2<1/\E[X]$. We
write
\[
\{\nu_1 r \leq N_r \leq\nu_2 r\} = \biggl\{
\sum_{1 \leq i \leq
\lceil\nu_1 r\rceil-1}X_i < r \biggr\} \cap \biggl\{
\sum_{1 \leq i \leq\lfloor\nu_2 r\rfloor
} X_i \geq r \biggr\}
\equalscolon E_r \cap F_r.
\]
Since $\nu\Lambda^\star(1/\nu)$ is continuous on $(1/b,1/a)$, by
\hyperref[thmpcramer]{Cram\'er's
theorem},
\[
\P\bigl[E_r^c\bigr]=e^{-\nu_1 r\Lambda^\star(1/\nu_1)(1+o(1))} \quad\mbox{and}\quad
\P[F_r]=e^{-\nu_2 r\Lambda^\star(1/\nu_2)(1+o(1))} ,
\]
except when $1/\nu_1=1/b$, in which case the expression for $\P
[E_r^c]$ becomes an upper bound.
Therefore,
\[
\P[E_r \cap F_r] = P[F_r] - \P
\bigl[F_r\cap E_r^c\bigr] = e^{-\nu_2 r\Lambda
^\star(1/\nu_2)(1+o(1))}
,
\]
which gives \eqref{eqnphow_many_summands}. The proof for the case
$1/\E[X]<\nu_1<\nu_2$ is analogous, and in the case $\nu_1<1/\E
[X]<\nu_2$,
both sides of \eqref{eqnphow_many_summands} are 0.
\end{pf}

%le2.4 #&#
\begin{lemma} \label{lempconvex} Suppose that $f$ is a convex
function on
$[a,b]\subseteq[0,\infty]$. Then $x\mapsto x f(1/x)$ is a convex
function on $[1/b,1/a]$.
\end{lemma}
\begin{pf}
Since $f$ is convex, it can be expressed as $ f(x) = \sup_i (\alpha
_i +
\beta_i x) $ for some pair of sequences of reals $\{\alpha_i\}_{i\in
\N}$
and $\{\beta_i\}_{i\in\N}$. For $x\in[0,\infty]$ we can write $ x
f(1/x) =
\sup_i (\alpha_i x + \beta_i)$, so it too is convex.
\end{pf}

%pr2.5 #&#
\begin{proposition} \label{propplimit-at-zero}
Let $X$ be a nonnegative real-valued random variable, and
let $\Lambda(\lambda)=\log\E[e^{\lambda X}]$.
Then
%
%e2.6 #&#
\begin{equation}
\lim_{\nu\downarrow0} \nu\Lambda^\star(1/\nu) = \sup\bigl\{
\lambda\dvtx \Lambda(\lambda)<\infty\bigr\}.
\end{equation}
\end{proposition}
\begin{pf}
Let $\lambda_0=\sup\{\lambda\dvtx \Lambda(\lambda)<\infty\}$, and note that
$0\leq\lambda_0\leq\infty$. Recall that $\Lambda^\star(x)
\colonequals
\sup_\mgfparam(\mgfparam x - \Lambda(\mgfparam))$, so
\[
\nu\Lambda^\star(1/\nu) = \sup_\mgfparam\bigl(
\mgfparam- \nu\Lambda (\mgfparam)\bigr).
\]
The supremum is not achieved for any $\lambda>\lambda_0$. If
$\lambda_0>0$, then $\E[X]<\infty$ and for $\nu\leq1/\E[X]$ the
supremum is achieved over the set $\lambda\geq0$
(\cite{DZ}, Lemma~2.2.5(b)). For any $\lambda\geq0$ we have
$\Lambda(\lambda)\geq0$, so $\nu\Lambda^\star(1/\nu) \leq
\lambda_0$ for
$0<\nu\leq1/\E[X]$. On the other hand, for any
$\lambda<\lambda_0$ we have $\Lambda(\lambda)<\infty$, so
$\liminf_{\nu\downarrow0}\nu\Lambda^\star(1/\nu) \geq\lambda$. Thus,
$\lim_{\nu\downarrow0} \nu\Lambda^\star(1/\nu) = \lambda_0$ when
$\lambda_0>0$.

Next, suppose $\lambda_0=0$. Then the supremum is achieved over the set
$\lambda\leq0$, for which $\Lambda(\lambda)\leq0$. For any
$\varepsilon>0$, there is a $\delta>0$ for which $-\varepsilon\leq
\Lambda(\lambda)\leq0$ whenever $-\delta\leq\lambda\leq0$. Since
$\lambda_0=0$, $\Pr[X=0]<1$, so $\Lambda(-\delta)<0$. Let $\nu_0 =
-\delta/ \Lambda(-\delta)$. By the convexity of $\Lambda$, for
$0\leq\nu
\leq\nu_0$, the supremum is achieved for $\lambda\in[-\delta,0]$. For
$\lambda$ in this range, $\lambda-\nu\Lambda(\lambda) \leq
\varepsilon
\nu$, so $0\leq\nu\Lambda^\star(1/\nu)\leq\varepsilon\nu$ when
$0<\nu\leq\nu_0$. Hence, $\lim_{\nu\downarrow0} \nu\Lambda
^\star(1/\nu) =
0$ when $\lambda_0=0$.
\end{pf}

We conclude by giving a parametrization of the graph of the function
$\gamma_\kappa$ over the interval $(0,\infty)$.

%pr2.6 #&#
\begin{proposition} \label{proppparameterization} Recall the
definition of
$\Lambda_\kappa$ in \eqref{eqnpLambdaR}. The graph of $\gamma
_\kappa$
over the interval $(0,\infty)$ is equal to the set
%
%e2.7 #&#
\begin{equation}
\label{eqpparam} \biggl\{ \biggl(\frac{1}{\Lambda_\kappa'(\lambda)},\lambda- \frac{\Lambda_\kappa(\lambda)}{\Lambda_\kappa'(\lambda)}
\biggr) \dvtx -\infty< \lambda < 1 - \frac{2}{\kappa}-\frac{3\kappa}{32} \biggr
\}.
\end{equation}
\end{proposition}

\begin{pf}
Recall that $\Lambda_\kappa^\star(x) = \sup_{\lambda\in\R}
[\lambda x -
\Lambda_\kappa(\lambda)]$. Since $\Lambda_\kappa'$ is continuous and
strictly increasing, the maximizing value of $\lambda$ for a given value
of $x$ is the unique $\lambda\in\R$ such that $\Lambda_\kappa
'(\lambda)
= x$. If we let $\lambda$ be this maximizing value, then we have
%
%e2.8 #&#
\begin{equation}
\label{eqpxlambdax} \Lambda_\kappa^\star(x) = \lambda x -
\Lambda_\kappa(\lambda).
\end{equation}
Differentiating \eqref{eqnpLambdaR} shows that as $\lambda$ ranges from
$-\infty$ to $1-2/\kappa-3\kappa/32$, $\Lambda_\kappa'(\lambda)$ ranges
from 0 to $\infty$. Using \eqref{eqpxlambdax} and writing $\nu=
1/x$, we
obtain
\[
\nu= \frac{1}{x} = \frac{1}{\Lambda_\kappa'(\lambda)}
\]
and
\[
\nu\Lambda_\kappa^\star(1/\nu) = \frac{1}{\Lambda_\kappa
'(\lambda)} \bigl(
\lambda \Lambda_\kappa'(\lambda) - \Lambda_\kappa(
\lambda) \bigr)= \lambda- \frac{\Lambda_\kappa(\lambda)}{\Lambda_\kappa'(\lambda)}.
\]
Therefore, $\{(\nu,\nu\Lambda_\kappa^\star(1/\nu)) \dvtx 0<\nu
<\infty\}$ is
equal to \eqref{eqpparam}.
\end{pf}

%pr2.7 #&#
\begin{proposition} \label{proppstrictly_convex}
The function $\gamma_\kappa$ is strictly convex over $[0,\infty)$.
\end{proposition}

\begin{pf}
Define $x(\lambda) = 1/\Lambda_\kappa'(\lambda)$ and $y(\lambda) =
\lambda- \Lambda_\kappa(\lambda)/\Lambda_\kappa'(\lambda)$. By
Proposition~\ref{proppparameterization}, the second derivative of
$\gamma_\kappa$
is given by
%
%e2.9 #&#
\begin{eqnarray}
\label{eqpsecond_derivative}&&\biggl ({\frac{d}{d\lambda} \biggl[\frac{y'(\lambda)}{x'(\lambda
)} \biggr]}\biggr)\Big/\bigl({x'(\lambda)}\bigr) \nonumber\\
&&\qquad= \biggl(8\pi^2 \sin^2 \biggl(\frac{\pi}{\kappa}\sqrt{8\kappa
\lambda+
(\kappa-4)^2} \biggr)\tan\biggl (\frac{\pi}{\kappa}\sqrt{8\kappa
\lambda+
(\kappa-4)^2} \biggr)\biggr)\\
&&\qquad\quad{}\Big/\biggl(2\pi\sqrt{8\kappa\lambda+ (\kappa-4)^2} -
\kappa\sin \biggl(\frac{2\pi}{\kappa}\sqrt{8\kappa\lambda
+(\kappa-4)^2} \biggr)\biggr).\nonumber
\end{eqnarray}
It is straightforward to confirm that $\sin^2t \tan t/(2t-\sin(2t))>0$
for all $t\in[0,\pi/2)$ (where we extend the definition to $t=0$ by
taking the limit of the expression as \mbox{$t\searrow0$}). Similarly,
$\sinh^2(2t)\tanh(t)/(\sinh(2t)-2t)>0$ for all $t\leq0$ (again extending
to $t=0$ by taking a limit). Setting
$t=\frac{\pi}{\kappa}\sqrt{8\kappa\lambda+ (\kappa-4)^2}$, these
observations imply that the second derivative of $\gamma_\kappa$ is
positive for all
$\lambda$ less than $1-2/\kappa-3\kappa/32$.
\end{pf}

%s2.2 #&#
\subsection{Overshoot estimates}
\label{subsecpgeneral}

Let $\{S_n\}_{n \in\N}$ be a random walk in $\R$ whose increments are
nonnegative and have exponential moments. In this section, we will
bound the tails of $S_n$ stopped at:
\begin{longlist}[(ii)]
\item[(i)] the first time that it exceeds a given threshold (Lemma~\ref
{lempfirsthitting}) and at
\item[(ii)] a random time which is stochastically dominated by a geometric
random variable (Lemma~\ref{lempgeodom}).
\end{longlist}

%le2.8 #&#
\begin{lemma}
\label{lempfirsthitting}
Suppose $\{X_j\}_{j\in\N}$ are nonnegative i.i.d.
random variables for which $\E[X_1] > 0$ and $\E[e^{\mgfparam_0
X_1}]<\infty$ for some $\mgfparam_0>0$.
Let $S_n=\sum_{j=1}^n X_j$
and $\tau_x =
\inf\{n \geq0 \dvtx S_n \geq x\}$. Then there exists $C>0$ (depending on
the law of $X_1$ and $\lambda_0$) such that $\P[S_{\tau_{x}} - x
\geq\alpha] \leq C\exp(-\mgfparam_0
\alpha)$ for all $x\geq0$ and $\alpha>0$.
\end{lemma}
\begin{pf}
Since $\E[X_1]>0$, we may choose $v>0$ so that $\P[X_1\geq v]\geq
\frac{1}2$.
We partition $(-\infty,x)$ into intervals of length $v$:
\[
(-\infty,x)=\bigcup_{k=0}^\infty
I_k \qquad\mbox{where } I_k=\bigl[x-(k+1) v,x-k v\bigr).
\]
Then we partition the event $S_{\tau_x}-x\geq\alpha$ into subevents:
\[
\{S_{\tau_x}-x\geq\alpha\} = \bigcup_{n=0}^\infty
\bigcup_{k=0}^\infty E_{n,k}
\qquad\mbox{where } E_{n,k}=\{ S_n \in I_k,
S_{n+1}\geq x+\alpha\}.
\]
The event $E_{n,k}$ implies $X_{n+1}\geq k v + \alpha$, and since
$X_{n+1}$ is independent of $S_n$, we have
\[
\P[E_{n,k}] \leq\P[S_n\in I_k] \times
\frac{\E[e^{\mgfparam_0
X}]}{e^{\mgfparam_0 (k v + \alpha)}}.
\]
On the event $S_n\in I_k$, since $X_{n+1}$ is independent of what
occurred earlier and is larger than $v$ with probability at least
$\frac{1}2$, we have $\P[S_{n+1}\in I_k|S_n\in I_k,S_{n-1},\ldots
,S_1]\leq\frac{1}2$. Thus,
\[
\sum_{n=0}^\infty\P[S_n\in
I_k] =\E\bigl[ \bigl|\{n \dvtx S_n \in I_k\}\bigr|
\bigr] \leq2.
\]
Thus,
\[
\P[S_{\tau_x}-x\geq\alpha] \leq\sum_{k=0}^\infty
\frac{2 \E[e^{\mgfparam_0 X}]}{e^{\mgfparam
_0(kv + \alpha)}} \leq\frac{2 \E[e^{\mgfparam_0 X}]}{1-e^{-\mgfparam_0 v}} \times e^{-\mgfparam_0\alpha}.%\qedhere
\]
\upqed\end{pf}

%le2.9 #&#
\begin{lemma}
\label{lempgeodom}
Let $\{X_j\}_{j\in\N}$ be an i.i.d. sequence of random variables
and let $S_n=\sum_{i=1}^n X_i$. Let $N$ be a positive integer-valued
random variable, which need not be independent of the $X_j$'s.
Suppose that there exists $\mgfparam_0>0$ for which $\E[e^{\mgfparam_0
X_1}]<\infty$ and $q \in(0,1)$ for which $\P[ N \geq k] \leq
q^{k-1}$ for every $k \in\N$. Then there exist constants $C,c>0$
(depending on $q$ and the law of $X_1$) for which
$\P[S_N>\alpha] \leq C \exp(-c \alpha)$ for every $\alpha> 0$.
\end{lemma}
\begin{pf}
Since $q \E[e^{\lambda X}]$ is a continuous function of $\lambda$
which is finite for $\lambda=\lambda_0>0$ and less than $1$ for
$\lambda=0$, there is some $c>0$ for which $q \E[e^{2 c X}]<1$. The
Cauchy--Schwarz inequality gives
\begin{eqnarray*}
\E\bigl[e^{c S_N}\bigr] &=& \E \Biggl[ \sum_{k=1}^{\infty}
e^{c S_k} \mathbf {1}_{\{N=k\}} \Biggr]
\\
&\leq&\sum_{k=1}^{\infty} \sqrt{\E
\bigl[e^{2 c S_k}\bigr] \P[{N=k}]} = q^{-1/2} \sum
_{k=1}^{\infty} \bigl(\sqrt{\E\bigl[e^{2 c X_1}
\bigr] q} \bigr)^k < \infty.
\end{eqnarray*}
We conclude using the Markov inequality $\P[S_N > \alpha] \leq e^{-c
\alpha} \E[e^{c S_N}]$.
\end{pf}

%s3 #&#
\section{CLE estimates}
\label{secpCLE}
In Section~\ref{subsecpcle_est}, we apply the estimates from
Section~\ref{secppreliminaries} to obtain asymptotic loop
nesting probabilities for CLE.
Then in Section~\ref{secpregularity_cle} we show that CLE
loops are well behaved in a certain sense that will allow
us to do the two-point estimates that we need for the Hausdorff dimension.

%s3.1 #&#
\subsection{CLE nesting estimates}
\label{subsecpcle_est}

We establish two technical estimates in this subsection.
Lemma~\ref{lempasymp_probability_disk} uses \hyperref[thmpcramer]{Cram\'er's
theorem} to compute
the asymptotics of the probability that the number $\Loopcount_z(\eps
)$ of CLE loops surrounding $B(z,\eps)$ has a certain
rate of growth as $\eps\to0$.

In preparation for the proof of Lemma~\ref{lempasymp_probability_disk}
below, we establish the continuity of the function $\gamma_\kappa$
defined in
\eqref{eqnpg(nu)}. Throughout, we let
$\lp_{\max}$ be the unique solution to $\gamma_\kappa(\nu) = 2$.

%pr3.1 #&#
\begin{proposition} \label{propplimit_at_zero}
The function $\gamma_\kappa$ is continuous. In particular,
%
%e3.1 #&#
\begin{equation}
\label{eqnplower_limit} \lim_{\lp\downarrow0} \gamma_\kappa(\lp) = 1 -
\frac{2}{\kappa} - \frac
{3\kappa}{32}.
\end{equation}
\end{proposition}

The quantity on the right-hand side of \eqref{eqnplower_limit} is two
minus the
almost-sure Hausdorff dimension of the $\CLE_\kappa$ gasket
\cite{SSW,NW,MSW}.

\begin{pf*}{Proof of Proposition~\ref{propplimit_at_zero}}
The continuity of $\gamma_\kappa$ on $(0,\infty)$ follows from
Proposition~\ref{propptransform_properties}(iv), and
the continuity at $0$ follows from Proposition~\ref{propplimit-at-zero}
and the fact that \eqref{eqnpLambdaR} blows up at $\lambda_0=1 -
\frac{2}{\kappa} - \frac{3\kappa}{32}$ but not for $\lambda
<\lambda_0$.
\end{pf*}

%le3.2 #&#
\begin{lemma}
\label{lempasymp_probability_disk}
Let $\kappa\in(8/3,8)$, $0 \leq\lp\leq\lp_{\mathrm{max}}$, and $0
< a
\leq b$. Then for all functions $\eps\mapsto\delta(\eps)$ decreasing
to 0 sufficiently slowly as $\eps\to0$ and for all proper simply
connected domains $D$ and points $z\in D$ satisfying $a \leq\confrad(z;D)
\leq b$, we have
%
%e3.2 #&#
\begin{equation}
\label{eqnpcount_estimate} \cases{ %
\displaystyle\lim
_{\eps\to0} \frac{\log\P [\lp\leq
\TLoopcount_z(\eps) \leq\lp+\delta(\eps) ]}{\log\eps} =\gamma_\kappa(\lp), &\quad
$\mbox{for }\lp> 0,$
\vspace*{2pt}\cr
\displaystyle\lim_{\eps\to0}\frac{\log\P [({1}/{2})\delta(\eps) \leq
\TLoopcount_z(\eps) \leq\delta(\eps)  ]}{\log\eps} = \gamma_\kappa (0),
& \quad$\mbox{for }\lp= 0,$}
\end{equation}
where $\gamma_\kappa$ is defined in \eqref{eqnpg(nu)}, and the
convergence is
uniform in the domain $D$.
\end{lemma}

\begin{pf}
Let $\{T_i\}_{i\in\N}$ \label{notpTi} be the sequence of $\log$
conformal radius increments associated with $z$. That is, defining
$U_z^i$ to be the connected component of $D \setminus\Loop_z^i$ which
contains $z$, we have $T_i \colonequals\log\confrad
(z;U_z^{i-1})-\log
\confrad(z;U_z^i)$.
Let
\[
S_n\colonequals\sum_{i=1}^n
T_i= \log\confrad(z;D) - \log\confrad \bigl(z;U_z^n
\bigr) \qquad\mbox{for } n \in\N.
\]
As in Corollary~\ref{corphow_many_summands}, we let $N_r=\min\{
n\dvtx S_n\geq r\}$.

By the Koebe one-quarter theorem
and the hypotheses of the lemma, we have
%
%e3.3 #&#
\begin{equation}
\label{eqnpcr_bounds} \log(a/4) - \log\inrad\bigl(z;U_z^n\bigr)
\leq S_n \leq\log b - \log\inrad\bigl(z;U_z^n
\bigr).
\end{equation}

Suppose first that $\lp>0$. Let
\begin{eqnarray*}
E &\colonequals&\bigl\{(\lp+\eta)\log(1/\eps) \leq N_{\log(a/4)+\log(1/\eps)} \bigr\}\quad
\mbox{and}
\\
F &\colonequals&\bigl\{N_{\log(b)+\log(1/\eps)} \leq(\lp +\delta_0-\eta)
\log(1/\eps)\bigr\}.
\end{eqnarray*}
It follows from \eqref{eqnpcr_bounds} that
for all fixed $\delta_0>0$ and $0<\eta<\delta_0/2$ and for all
$\eps=\eps(\eta)>0$ sufficiently small, we have
\[
\bigl\{\lp\leq\TLoopcount_z(\eps) \leq\lp+\delta_0\bigr
\} \supset E \cap F.
\]
By Corollary~\ref{corphow_many_summands}, $\log\P[E]/\log\eps\to
\inf_{\xi\in[\nu+\eta,\nu+\delta_0-\eta]}\gamma_\kappa(\xi
)$. Furthermore,\break 
\hyperref[thmpcramer]{Cram\'er's
theorem} implies that $\P[F \given E]=\eps^{o(1)}$. It follows
that
\[
\liminf_{\eps\to0} \frac{\log\P [\lp\leq
\TLoopcount_z(\eps) \leq\lp+\delta_0 ]}{\log\eps} \geq \inf
_{\xi\in[\lp+\eta,\lp+\delta_0-\eta]}\gamma_\kappa(\xi).
\]
Letting $\eta\to0$ and using an analogous argument to upper bound the
limit supremum of the quotient on the left-hand side, we find that
\[
\lim_{\eps\to0} \frac{\log\P [\lp\leq
\TLoopcount_z(\eps) \leq\lp+\delta_0 ]}{\log\eps} = \inf_{\xi\in[\lp,\lp+\delta_0]}
\gamma_\kappa(\xi).
\]
By the continuity of $\gamma_\kappa$ on $[0,\infty)$, we may choose
$\delta(\eps)\downarrow0$ so that \eqref{eqnpcount_estimate} holds.
The proof for $\lp=0$ is similar. As above, we show that for
$\delta_0>0$ fixed, we have
\[
\lim_{\eps\to0} \frac{\log\P [\delta_0/2 \leq
\TLoopcount_z(\eps) \leq\delta_0  ]}{\log\eps} = \inf_{\xi\in[\delta_0/2,\delta_0]}
\gamma_\kappa(\xi).
\]
Again, choose $\delta(\eps)\downarrow0$ so that
\eqref{eqnpcount_estimate} holds.
\end{pf}

%s3.2 #&#
\subsection{Regularity of CLE}
\label{secpregularity_cle}

Let $D \subsetneq\C$ be a proper simply connected domain
and let $\Gamma$ be a $\CLE_\kappa$ in $D$. For each $z \in D$, let
$\Loop_z^j$ be the $j$th largest loop of $\Gamma$ which surrounds $z$.
In this section, we estimate the tail behavior of the number of
such loops $\Loop_z^j$ which intersect the boundary of a ball $B(z,r)$
in $D$.

%le3.3 #&#
\begin{lemma}
\label{lempno_intersection}
For each $\kappa\in(8/3,8)$
there exists $p_1 = p_1(\kappa) > 0$ such that
for any proper simply connected domain $D$ and $z \in D$,
\[
\P\bigl[ \Loop_z^1 \cap\partial D = \varnothing\bigr]
\geq p_1 > 0.
\]
\end{lemma}
\begin{pf}
If $\kappa\in(8/3,4]$, we can take $p_1 = 1$ since the loops of such
$\CLE$s almost surely do not intersect the boundary of $D$. Assume
$\kappa\in(4,8)$. By the conformal invariance of $\CLE$, the boundary
avoidance probability is independent of the domain $D$ and the point $z$,
so we take $D = \D$. Let $\eta= \eta^0$ be the branch of
the exploration process of $\Gamma$ targeted at $0$ and let $(W,V)$ be
the driving pair for $\eta$. Let $\tau$ be an almost surely positive and
finite stopping time such that $\eta|_{[0,\tau]}$ almost surely does not
surround 0 and $\eta(\tau) \neq V_\tau$ almost surely. Then
$\eta|_{[\tau,\infty)}$ evolves as an ordinary chordal $\SLE_\kappa$
process in the connected component of $\D\setminus\eta([0,\tau])$
containing $0$ targeted at $V_\tau$, up until disconnecting $V_\tau$ from
$0$. In particular, $\eta|_{(\tau,\infty)}$ almost surely intersects the
right-hand side of $\eta|_{[0,\tau]}$ before surrounding $0$. Since $\eta
$ is
almost surely not space filling \cite{RS05} and cannot trace itself, this
implies that, almost surely, there exists $z \in\Q^2 \cap\D$ such
that the probability that $\eta$ makes a clockwise loop around $z$ before
surrounding $0$ is positive. This in turn implies that with
positive probability, the branch $\eta^z$ of the exploration tree
targeted at $z$ makes a clockwise loop around $z$ before making a
counterclockwise loop around~$z$. By~\cite{SHE_CLE}, Lemma~5.2, this
implies $\P[ \Loop_z^1 \cap\partial\D= \varnothing] > 0$.
\end{pf}

Suppose that $D = \D$. By the continuity of $\CLE_\kappa$ loops,
Lemma~\ref{lempno_intersection}
implies there exists $r_0 = r_0(\kappa) < 1$ such that
%
%e3.4 #&#
\begin{equation}
\label{eqnpp1over2} \P\bigl[ \Loop_0^1 \subset
B(0,r_0)\bigr] \geq\frac{p_1}{2}.
\end{equation}

%le3.4 #&#
\begin{lemma}
\label{lemploop_contain_prop}
For each $\kappa\in(8/3,8)$ there exists $p = p(\kappa) > 0$
such that
for any proper simply connected domain $D$ and $z \in D$,
\[
\P\bigl[\Loop_z^2 \subseteq B\bigl(z,\dist(z,\partial D)
\bigr)\bigr] \geq p. % \mathclap{\smash{\raisebox{-30pt}{
\]
\end{lemma}

\includegraphics{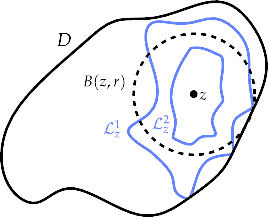}

% \caption{\qq{}.}\label{fig7}

\begin{pf}
Let $D_1$ be the connected component of $D \setminus\Loop_z^1$ which
contains $z$ and let $X = \confrad(z;D_1) / \confrad(z;D) \leq1$. Let
$\varphi\colon\D\to D_1$ be a conformal map with $\varphi(0)=z$, and
let $r=\dist(z,\partial D)$. By the Koebe one-quarter theorem,
we have $\confrad(z;D) \leq4r$, hence
\[
\bigl|\varphi'(0)\bigr| = \confrad(z;D_1) = \confrad(z;D) \cdot
\frac{\confrad(z;D_1)}{\confrad(z;D)} \leq4 X r.
\]
The variant of the Koebe distortion theorem which bounds $|f(z)-z|$
(see, e.g.,~\cite{Law05}, Proposition~3.26)
then implies for $|w| < r_0 < 1$, we have
%
%e3.5 #&#
\begin{equation}
\label{eqnpdistortion} \bigl|\varphi(w) - z\bigr| \leq4Xr \frac{r_0}{(1-r_0)^2}.
\end{equation}
Since the distribution of $-\log X$ has a positive density on
$(0,\infty)$ \cite{SSW}, the
probability of the event $E = \{X \leq(1-r_0)^2/(4 r_0) \}$ is
bounded below by $p_2=p_2(\kappa) > 0$ depending only $\kappa$.
On $E$, the right-hand side of \eqref{eqnpdistortion} is bounded
above by $r$,
that is, $\varphi(r_0 \D) \subset B(z,r)$.
By the conformal invariance and renewal property of CLE, the loop
$\Loop_z^2$ in $D$
is distributed as the image under $\varphi$ of the loop $\Loop_0^1$
in $\D$, which is independent of $X$.
Thus, by \eqref{eqnpp1over2},
$\P[ \Loop_z^2 \subseteq B(z,r)] \geq\P[E]\P[ \Loop_z^2 \subseteq
B(z,r) \given E] \geq(p_2)(p_1/2) \equalscolon p >0$.
\end{pf}

For the $\CLE_\kappa$ $\Gamma$ in $D$, $z\in D$ and $r>0$ we define
\begin{subequations}
\label{eqpJcapsubset} \label{notpJcapsubset}
%
%e3.6 #&#
%e3.7 #&#
\begin{eqnarray}
J^\cap_{z,r}&\colonequals&\min\bigl\{j \geq1 \dvtx
\Loop_z^j \cap B(z,r) \neq\varnothing\bigr\},
\\
J^\subset_{z,r}&\colonequals&\min\bigl\{j \geq1 \dvtx
\Loop_z^j \subset B(z,r) \bigr\}.
\end{eqnarray}
\end{subequations}

%co3.5 #&#
\begin{corollary}
\label{corploop_contain_stoch_dom}
$J^\subset_{z,r} - J^\cap_{z,r}$ is
stochastically dominated by $2\wt{N}$ where $\wt{N}$ is a geometric random
variable with parameter $p = p(\kappa) > 0$ which depends only on
$\kappa\in(8/3,8)$.
\end{corollary}
\begin{pf}
Immediate from Lemma~\ref{lemploop_contain_prop} and the renewal
property of $\CLE_\kappa$.
\end{pf}

%le3.6 #&#
\begin{lemma} \label{lempannulus-loop} For each $\kappa\in(8/3,8)$ there
exist $c_1 > 0$ and $c_2>0$ such that for any proper simply connected
domain $D$ and point $z\in D$, for any positive numbers $r$ and $R$ for
which $r<R$ and $B(z,R)\subset D$, a $\CLE_\kappa$ in $D$ contains a loop
$\Loop$ surrounding $z$ for which $\Loop\subset B(z,R)$ and $\Loop
\cap
B(z,r)=\varnothing$ with probability at least $1-(c_1 r/R)^{c_2}$.
\end{lemma}
\begin{pf}
For convenience, we let $x=\log(R/r)$ and rescale so that $R=1$.
For the $\CLE_\kappa$ $\Gamma$, let $\lambda_j=-\log\confrad
(\Loop_z^j(\Gamma))$.
By the renewal property of $\CLE_\kappa$, $\{\lambda_{j+1}-\lambda
_j\}$ form an i.i.d. sequence,
and their distribution has exponential tails~\cite{SSW}.
Now $\min(\{\lambda_j\} \cap(0,\infty))=\lambda_{J^\cap_{z,1}}$, which
by Lemma~\ref{lempfirsthitting}
is dominated by a distribution which has exponential tails
and depends only on $\kappa$. By \hyperref[thmpcramer]{Cram\'er's
theorem}, there
is a constant $c > 0$ so that $\lambda_{J^\cap_{z,1} + c x} \leq x -
\log4$ except with
probability exponentially small in $x$.
By Corollary~\ref{corploop_contain_stoch_dom},
$J^\subset_{z,1} - J^\cap_{z,1}$ is stochastically dominated by twice
a geometric random variable,
and so $J^\subset_{z,1} \leq J^\cap_{z,1} + c x$ except with
probability exponentially small in $x$.
If both of these high probability events occur, then $\Loop_{J^\subset
_{z,1}}\cap B(z,e^{-x}) = \varnothing$.
\end{pf}

%le3.7 #&#
\begin{lemma} \label{lempconditional-overshoot}
Let $X$ be a random variables whose law is the difference in log
conformal radii of successive $\CLE_\kappa$ loops.
Let $f_M$ denote the density function of $X-M$ conditional on $X\geq M$.
For some constant $C_\kappa$ depending only on $\kappa$,
\[
\sup_M f_M \leq C_\kappa\times
\exp \bigl[-(1-2/\kappa- 3\kappa/32)x \bigr].
\]
For all $M$ and all $x>1$,
the actual density is within a constant factor of this upper bound.
\end{lemma}
\begin{pf}
The density function for the law of $X$ is \cite{SSW}, equation (4)
\[
-\frac{\kappa\cos(4\pi/\kappa)}{4\pi}\sum_{j=0}^\infty(-1)^j
\biggl(j+\frac{1}2\biggr) \exp \biggl[-\frac{(j+{1}/{2})^2-(1-
{4}/{\kappa})^2}{8/\kappa}x \biggr].
\]
For large enough $x$, the first term dominates the sum of the other
terms. For small~$x$, a different formula (\cite{SSW}, Theorem~1), implies
that the density is bounded by a constant. Integrating, we obtain
$\P[X\geq M]$ to within constants,
and then obtain the conditional probability to within constants.
\end{pf}

%s4 #&#
\section{Nesting dimension}
\label{secpnesting-dimension}

In this section, we prove Theorem~\ref{thmpmain}, which gives the Hausdorff
dimension of the set $\Phi_\lp(\Gamma)$ for a $\CLE_\kappa$
$\Gamma$ in a
simply connected proper domain $D\subsetneq\C$. Define
\begin{eqnarray*}
\Phi^+_{\lp}(\Gamma) &\colonequals& \Bigl\{z\in D\dvtx \liminf
_{r\to
0}\TLoopcount_z(r;\Gamma)\geq\lp \Bigr\},
\\
\Phi^-_{\lp}(\Gamma) &\colonequals& \Bigl\{z\in D\dvtx \limsup
_{r\to
0}\TLoopcount_z(r;\Gamma)\leq\lp \Bigr\}.
\end{eqnarray*}
Then the sets $\Phi^{\pm}_{\lp}(\Gamma)$ are monotone in $\lp$, and
$\Phi_{\lp}(\Gamma) = \Phi^+_{\lp}(\Gamma) \cap\Phi^-_{\lp
}(\Gamma)$. (We
suppress $\Gamma$ from the notation when it is clear from context.)

%pr4.1 #&#
\begin{proposition} \label{proppPhi-conformal} $\Phi^+_{\lp}(\Gamma
)$ and
$\Phi^-_{\lp}(\Gamma)$ are invariant under conformal maps.
\end{proposition}

Conformal invariance of these CLE exceptional points is easier to prove
than conformal invariance of the thick points of the
Gaussian free field (\cite{HMP}, Corollary~1.4).
\begin{pf}
Let $\varphi\colon D \to D'$ be a conformal map, and let $\Gamma$ be a
$\CLE_\kappa$ in $D$; $\varphi(\Gamma)$ is a $\CLE_\kappa$ in
$D'$. By
the Koebe distortion theorem, for all $\eps> 0$ small enough
\[
\Loopcount_{z} \bigl(16 \eps\bigl|\varphi'(z)\bigr|^{-1};
\Gamma \bigr) \leq \Loopcount_{\varphi(z)}\bigl(\eps;\varphi(\Gamma)\bigr) \leq
\Loopcount_{z} \bigl(\tfrac{1}{16}\eps\bigl|\varphi'(z)\bigr|^{-1};
\Gamma \bigr).
\]
But
\begin{eqnarray*}
\liminf_{\eps\to0^+} \frac{\Loopcount_z(16^{\pm1}\eps|\varphi'(z)|^{-1};\Gamma)}{\log
(1/\eps)} & =& \liminf
_{\eps\to0^+} \frac{\Loopcount_z(\eps;\Gamma)}{\log
(1/(16^{\mp1}\eps|\varphi'(z)|))}
\\
&= &\liminf_{\eps\to0^+} \frac{\Loopcount_z(\eps;\Gamma)}{\log
(1/\eps)}.
\end{eqnarray*}
Thus,
\[
\liminf_{\eps\to0^+} \TLoopcount_z(\eps;\Gamma) =
\liminf_{\eps
\to0^+} \TLoopcount_{\varphi(z)}\bigl(\eps;\varphi(
\Gamma)\bigr) ,
\]
so $\varphi(\Phi^+_{\nu}(\Gamma)) = \Phi^+_{\nu}(\varphi(\Gamma
))$. Similarly,
$\varphi(\Phi^-_{\lp}(\Gamma)) = \Phi^-_{\nu}(\varphi(\Gamma))$.
\end{pf}

Observe that conformal maps preserve Hausdorff dimension: away from the
boundary, conformal maps are bi-Lipschitz, and the Hausdorff dimension
of a
countable union of sets is the maximum of the Hausdorff dimensions. So we
may restrict our attention to the case where the domain $D$ is the unit
disk $\D$.

%s4.1 #&#
\subsection{Upper bound}
\label{secpupperbound}

Let $\Gamma$ be a $\CLE_\kappa$ in $\D$. Here, we upper bound the Hausdorff
dimension of $\Phi^{\pm}_{\lp}(\Gamma)$. Recall that $\gamma
_\kappa$ is
defined in
\eqref{eqnpg(nu)} and that $\lp_{\mathrm{max}}$ is the unique value
of $\lp
\geq0$ such that $\gamma_\kappa(\lp) = 2$. Moreover, $\gamma
_\kappa(\lp) \in[0,2)$ for
$0 \leq
\lp< \lp_{\mathrm{max}}$.

%pr4.2 #&#
\begin{proposition}
\label{proppupper_bound}
If $0 \leq\lp\leq\lp_{\mathrm{typical}}$, then $\dim_\CH\Phi
^-_{\lp}(\CLE_\kappa)
\leq2-\gamma_\kappa(\lp)$ almost surely.
If $\lp_{\mathrm{typical}}\leq\lp\leq\lp_{\max}$, then $\dim
_\CH\Phi^+_{\lp}(\CLE
_\kappa) \leq2-\gamma_\kappa(\lp)$ almost surely.
If $\lp> \lp_{\max}$, then $\Phi^+_{\lp}(\CLE_\kappa) =
\varnothing$ almost surely.
\end{proposition}

\begin{pf}
Observe that the unit disk can be written as a countable union of
M\"obius transformations of $B(0,1/2)$. For example, for $q \in\D\cap
\Q^2$, define $\varphi_q$ to be the Riemann map for which
$\varphi_q(0)=q$ and $\varphi_q'(0)>0$. Then $\D= \bigcup_{q \in\D
\cap
\Q^2} \varphi_q(B(0,1/2))$. By M\"obius invariance, therefore, it
suffices to bound the Hausdorff dimension of $\Phi^{\pm}_{\lp} \cap
B(0,1/2)$ for a $\CLE_\kappa$ in $\D$. We will prove the result for
$\Phi^{+}_{\lp} \cap B(0,1/2)$, as $\Phi^{-}_{\lp} \cap B(0,1/2)$ is
similar.

To upper bound the Hausdorff dimension, it suffices to find good covering
sets. Let $r > 0$. Let $\CD^r$ be the set of open balls in $\C$ which
are centered at points of $r \Z^2 \cap B(0,1/2+r/\sqrt{2})$ and have
radius $(1+1/\sqrt{2})r$. For every point $z\in B(0,1/2)$, the closest
point in $r\Z^2$ to $z$ is the center of a ball $U\in\CD^r$ for which
$B(z,r)\subset U \subset B(z,(1+\sqrt{2})r)$.

For each ball $U \in\CD^r$, let $z(U)$ be the center of $U$.
We define
%
%e4.1 #&#
\begin{equation}
\label{eqnpUdef} \CU^{r,\lp+} \colonequals \bigl\{U \in\CD^r
\dvtx \TLoopcount_{z(U)}(r) \geq\lp \bigr\}. % \\
% \CU^{r,\lp-} &\colonequals\left\{U \in\CD^r:
% \TLoopcount_{z(U)}(r) \leq\lp\right\}.
\end{equation}

The conformal radius of $\D$ with respect to $z\in\D$ is $1-|z|^2$.
For $U
\in\CD^{r}$, we have $|z(U)|\leq1/2+r/\sqrt{2}$, so $\frac{1}2\leq
\confrad(z;\D) \leq1$ provided $r\leq1-1/\sqrt{2}$. Thus by
\hyperref[thmpcramer]{Cram\'er's theorem} (as in the proof of
Lemma~\ref{lempasymp_probability_disk}) and the continuity of
$\gamma_\kappa(\lp)$, for $\lp>\lp_{\mathrm{typical}}$ we have
\[
\P\bigl[U \in\CU^{r,\lp+}\bigr] \leq r^{\gamma_\kappa(\lp) + o(1)},
\]
%
% and for $\lp<\lptyp$ we have
% \[ \P[U \in\CU^{r,\lp-}] \leq r^{\g(\lp) + o(1)} , \]
where for fixed $\lp$, the $o(1)$ term tends to $0$ as $r\to0$,
uniformly in $U$.

Next, we define
%
%e4.2 #&#
\begin{equation}
\label{eqnpcover_def} \CC^{m,\lp+} \colonequals\bigcup
_{n \geq m} \CU^{\exp(-n),\lp+}.
\end{equation}

Suppose that $z \in\Phi^+_{\lp}(\Gamma)\cap B(0,1/2)$. Since
$\liminf_{\eps\to0} \TLoopcount_z(\eps) \geq\lp$, for any $\lp
'<\lp$, for
all large enough $n$, $\Loopcount_z((1+\sqrt{2})e^{-n}) \geq\lp' n$.
There is a ball $U\in\CU^{e^{-n}}$ for which $U\subset B(z,(1+\sqrt{2})
e^{-n})$, and so $\Loopcount_{z(U)}((1+1/\sqrt{2})e^{-n}) \geq
\Loopcount_z((1+\sqrt{2})e^{-n})$, so $U \in\CU^{e^{-n},\lp'+}$. Hence,
for any $m\in\N$ and $\lp'<\lp$, we conclude that $\CC^{m,\lp'+}$
is a cover
for $\Phi^+_{\lp}(\Gamma) \cap B(0,1/2)$.

% Suppose that $z \in\Phi^-_{\lp}(\Gamma)\cap B(0,1/2)$.
% Since $\limsup_{\eps\to0} \TLoopcount_z(\eps) \leq\lp$,
% for any $\lp'>\lp$, for all large enough $n$, $\Loopcount_z(e^{-n})
%\leq\lp' n$.
% There is a ball $U\in\CU^{e^{-n}}$ for which
% $B(z,e^{-n})\subset U$, and so $\Loopcount_z(e^{-n}) \geq
%\Loopcount_{z(U)}((1+1/\sqrt{2})e^{-n})$,
% so $U \in\CU^{e^{-n},\lp'-}$. Hence, for any $m\in\N$ and $\lp'>
%\lp$, we have
% that $\CU^{m,\lp'-}$ is a cover for $\Phi^-_{\lp}(\Gamma)$.

We use this cover to bound the $\alpha$-Hausdorff measure of $\Phi
^{+}_{\lp}(\Gamma)$.
If $m\in\N$ and $\lp'>\lp>\lp_{\mathrm{typical}}$,
%
%e4.3 #&#
\begin{eqnarray}
\label{eqpdim_upper_bound} \E\bigl[\CH_\alpha\bigl(\Phi^{+}_{\lp}(
\Gamma)\bigr)\bigr] &\leq &\E \biggl[\sum_{U \in\CC^{m,\lp'+}} \bigl(
\diam(U)\bigr)^\alpha \biggr]
\nonumber\\
&= &\sum_{n \geq m} \sum
_{U \in\CD^{e^{-n}}} \bigl[(2+\sqrt{2})e^{-n}
\bigr]^\alpha\P\bigl[U \in\CU^{e^{-n},\lp
'+}\bigr]
\\
\nonumber
&\leq&\sum_{n \geq m} e^{n(2-\alpha-\gamma_\kappa(\lp')+o(1))}.
\end{eqnarray}
If $\alpha>2-\gamma_\kappa(\lp')$, the right-hand side tends to $0$
as $m\to
\infty$.
Since $m$ was arbitrary, we conclude that
$\E[\CH_\alpha(\Phi^{+}_{\lp}(\Gamma))] =0$. Therefore, almost surely
$\CH_\alpha(\Phi^{+}_{\lp}(\Gamma))=0$. Any such $\alpha$ is an
upper bound
on $\dim_\CH\Phi^{+}_{\lp}(\Gamma)$. The continuity of $\gamma
_\kappa(\lp)$ then
implies that almost surely $\dim_\CH\Phi^+_{\lp}(\Gamma)\leq
2-\gamma_\kappa(\lp)$
when $\lp>\lp_{\mathrm{typical}}$.
% and $\dimH\Phi^-_{\lp}(\Gamma)\leq
% 2-\g(\lp)$ (when $\lp<\lptyp$).
When $\lp=\lp_{\mathrm{typical}}$, the dimension bound (which is
$2$) holds trivially.
Finally, when $\lp>\lp_{\max}$, the bound in \eqref{eqpdim_upper_bound}
shows that $\CH_0(\Phi^+_{\lp}(\Gamma))=0$ almost surely. Therefore,
$\Phi^+_{\lp}(\Gamma)=\varnothing$ almost surely.
\end{pf}

%s4.2 #&#
\subsection{Lower bound}
\label{secplowerbound}

Next, we lower bound $\dim_\CH(\Phi_\lp(\Gamma))$. As we did for
the upper
bound, we assume without loss of generality that $D=\D$. We will
follow the
strategy used in \cite{HMP} for GFF thick points: we introduce a subset
$P_\lp(\Gamma)$ of $\Phi_\lp(\Gamma)$ which has the property that the
number and geometry of the loops which surround points in $P_\lp
(\Gamma)$
are controlled at every length scale. This reduction is useful because the
correlation structure of the loop counts for these special points is easier
to estimate (Proposition~\ref{proppnearind}) than that of arbitrary points
in $\Phi_\lp(\Gamma)$. Then we prove that the Hausdorff dimension of this
special class of points is at least $2-\gamma_\kappa(\lp)$ with
positive probability.
We complete the proof of the almost sure lower bound of
$\dim_\CH\Phi_\lp(\Gamma)$ using a zero--one argument.

%le4.3 #&#
\begin{lemma} \label{lempapprox-annulus}
Let $\Gamma$ be a $\CLE_\kappa$
in the unit disk $\D$, and fix $\lp\geq0$. Then for functions
$\delta(\eps)$ converging to 0 sufficiently slowly as $\eps\to0$ and
for sufficiently large $M>1$, the event that:

%{
%\centering
%%
%\begin{minipage}{.66\textwidth}
%
\begin{longlist}[(ii)]
\item[(i)]
there is a loop which is contained in the annulus $\ol{B(0,\eps
)}\setminus B(0,\varepsilon/M)$ and which surrounds
$B(0,\eps/M)$, and
\item[(ii)] the index $J$ of the outermost such loop in the annulus $\ol
{B(0,\varepsilon)} \setminus B(0,\varepsilon/M)$
satisfies $ \lp\log\eps^{-1} \leq J \leq(\lp+ \delta(\eps))
\log\eps^{-1}$,
\end{longlist}
has probability at least $\eps^{\gamma_\kappa(\lp)+o(1)}$ as $\eps
\to0$.
%\end{minipage}

\includegraphics{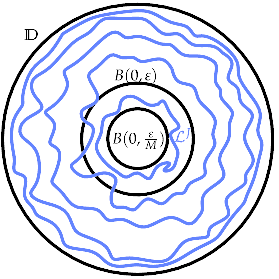}

%\hfill
%\begin{minipage}{.34\textwidth}
%
% \includegraphics[width=.82\textwidth]{figures/annulus-loop}
%\end{minipage}
\end{lemma}
\begin{pf}
We define $\delta(\eps)$ to be $2$ times the function denoted $\delta$
in Lemma~\ref{lempasymp_probability_disk}. Let $E_1$ denote the event
that between $\lp\log\eps^{-1}$ and $(\lp+\frac{1}{2}\delta(\eps))
\log\eps^{-1}$ loops surround $B(z,\eps)$, let $E_2$ denote the event
that at most $\frac{1}{2}\delta(\eps) \log\eps^{-1}$ loops intersect
the circle $\partial B(z,\eps)$, and let $E_3$ denote the event that
there is a loop winding around the closed annulus
$\overline{B(0,\eps)}\setminus B(0,\eps/M)$.

Lemma~\ref{lempasymp_probability_disk} implies
%
%e4.4 #&#
\begin{equation}
\label{eqnpPr-E1} \P[E_1] = \eps^{\gamma_\kappa(\lp)+o(1)}\qquad \mbox{as }\eps\to0.
\end{equation}
Corollary~\ref{corploop_contain_stoch_dom} implies that for sufficiently
small $\eps$, we have
%
%e4.5 #&#
\begin{equation}
\label{eqnpe2_given_e1_bound} \P[E_2 \given E_1] \geq\tfrac{3}{4}.
\end{equation}

Lemma~\ref{lempconditional-overshoot}
applied to the log conformal radius increment sequence implies that for
some large enough $M$,
%
%e4.6 #&#
\begin{equation}
\label{eqnpCR-cap-overshoot} \P \bigl[\confrad \bigl(0;U_0^{J_{0,\eps}^\cap} \bigr)
\geq M^{-1/2} \eps \given E_1 \bigr] \geq\tfrac{7}{8}.
\end{equation}
Lemma~\ref{lempgeodom} and Corollary~\ref
{corploop_contain_stoch_dom} together imply that
for large enough $M$
%
%e4.7 #&#
\begin{equation}
\label{eqnpCR-subset-cap-overshoot} \P \bigl[\confrad \bigl(0;U_0^{J_{0,\eps}^\subset} \bigr)/
\confrad \bigl(0;U_0^{J_{0,\eps}^\cap} \bigr) \geq M^{-1/2}
\given E_1 \bigr] \geq\tfrac{7}{8}.
\end{equation}
Combining \eqref{eqnpPr-E1}, \eqref{eqnpe2_given_e1_bound},
\eqref{eqnpCR-cap-overshoot} and \eqref
{eqnpCR-subset-cap-overshoot}, we arrive at
\[
\P[ E_1\cap E_2\cap E_3] =
\eps^{\gamma_\kappa(\lp)+o(1)}\qquad \mbox {as } \eps\to 0.
\]
The event $E_1\cap E_2 \cap E_3$ implies the event described in the lemma.
\end{pf}

We define the set $P_\lp=P_\lp(\Gamma)$ as follows.
For $z\in\D$ and $k\geq0$, we inductively define (see Figure~\ref
{figpperfect_definition}):
\begin{itemize}
\item Let $\tau_0=0$.
\item Let $V_z^k=U_{z}^{\tau_k}$ be the connected
component of $\D\setminus\Loop_z^{\tau_k}$ containing $z$.
In particular, $V_z^0=D=\D$.
\item Let $\varphi_z^k$ be the conformal map from $V_z^k$ to $\D$
with $\varphi_z^k(z) = 0$ and $(\varphi_z^k)'(z) > 0$.
\item Let $t_k = 2^{-(k+1)}$. \label{notptk}
\item Let $\tau_{k+1}$ be the smallest $j \in\N$ such that $\varphi
_z^k(\Loop_z^j)\subset\overline{B(0,t_k)}$.
\end{itemize}

Let $\wt{\Gamma}_z^k$ be the image under $\varphi_z^k$ of
the loops of $\Gamma$ which are surrounded by $\Loop_z^{\tau_k}$ and in
the same component of $\D\setminus\Loop_z^{\tau_k}$ as $z$.
Then $\wt{\Gamma}_z^k$ is a $\CLE_\kappa$ in $\D$.

Let $M>1$ be a large enough constant for Lemma~\ref
{lempapprox-annulus}, and let
$E_z^k$ to be the event described in
Lemma~\ref{lempapprox-annulus} for the $\CLE$ $\wt{\Gamma}_z^k$ and
$\eps=t_k$. We define
\[
E_z^{k_1,k_2}\colonequals\bigcap_{k_1\leq k < k_2}
E_z^{k}.
\]

%f7 #&#
\begin{figure}

\includegraphics{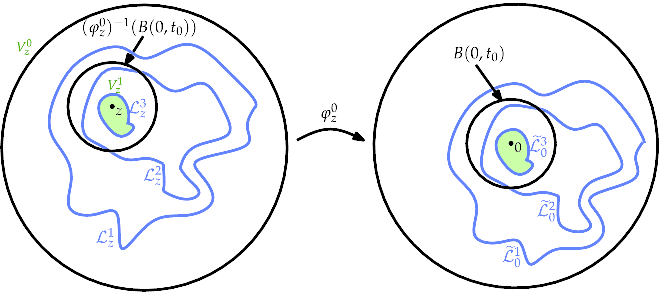}

\caption{First step in the construction of the set of ``perfect
points'' $P_\lp(\Gamma)$ of the CLE $\Gamma$.}
\label{figpperfect_definition}
\end{figure}

Throughout the rest of this section, we let \label{notpsk}
%
%e4.8 #&#
\begin{equation}
\label{eqnps_k_definition} s_k = \prod_{0 \leq i < k}
t_i \qquad\mbox{for } k \geq0.
\end{equation}

%le4.4 #&#
\begin{lemma}
\label{lempek_facts}
There exist sequences $\{r_k\}_{k\in\N}$ and $\{R_k\}_{k\in\N}$
satisfying
%
%e4.9 #&#
\begin{equation}
\label{eqprR} \lim_{k\to\infty}\frac{\log r_k }{\log s_k} = \lim
_{k\to\infty}\frac{\log R_k }{\log s_k} = 1
\end{equation}
such that for all $z\in\overline{B(0,1/2)}$ and $k\geq0$, we have
%
%e4.10 #&#
\begin{equation}
\label{eqpVzk} B(z,r_k) \subset V_z^k
\subset B(z,R_k)
\end{equation}
on the event $E_z^{0,k}$.
\end{lemma}
\begin{pf}
For $0<j\leq k$, the chain rule implies that on the event
$E_z^{0,k}$ we have
%
%e4.11 #&#
\begin{equation}
\label{eqnpconf_rad_increasing} \confrad\bigl(z;V_z^{j}\bigr) = \confrad
\bigl(0;\varphi_z^{j-1}\bigl(V_z^{j}
\bigr)\bigr) \confrad \bigl(z;V_z^{j-1}\bigr) \leq
t_{j-1} \confrad\bigl(z;V_z^{j-1}\bigr) ,
\end{equation}
where the inequality follows from the Schwarz lemma.
Iterating the inequality in~\eqref{eqnpconf_rad_increasing}, we see that
%
%e4.12 #&#
\begin{eqnarray}\label{eqnpconf_rad_chain}
\confrad\bigl(z;V_z^k\bigr) &\leq& t_{k-1}
\confrad\bigl(z;V_z^{k-1}\bigr) \leq\cdots
\leq(t_{k-1} \cdots t_0) \confrad\bigl(z;V_z^0
\bigr)
\nonumber
\\[-8pt]
\\[-8pt]
\nonumber
&=& s_k \confrad(z;\D).
\end{eqnarray}
Since $|((\varphi_z^{k-1})^{-1})'(0)| = \confrad(z;V_z^{k-1})$, it follows
from the
Koebe distortion theorem
that $V_z^k \subseteq
B(z,t_{k-1}/(1-t_{k-1})^2 \confrad(z;V_z^{k-1}))$.
Since $\confrad(z;V_z^{k-1})\leq s_{k-1}\confrad(z;\D)$, $\confrad
(z;\D)=1-|z|^2\leq1$, and
$t_{k-1} \leq1/2$, we see from \eqref{eqnpconf_rad_chain} that
$V_z^k \subseteq B(z,4 s_k)$,
so we set $R_k = 4 s_k$ to get the second inclusion in \eqref{eqpVzk}.

To find $\{r_k\}_{k\in\N}$ satisfying the first inclusion in
\eqref{eqpVzk}, we observe that on $E_z^{0,k}$ we have
\begin{eqnarray*}
\confrad\bigl(z;V_z^k\bigr) &\geq& M^{-1}
t_{k-1} \confrad\bigl(z;V_z^{k-1}\bigr) \geq \cdots
\geq M^{-k} (t_{k-1} \cdots t_0) \confrad
\bigl(z;V_z^0\bigr)
\\
&=& M^{-k} s_k \confrad(z;\D).
\end{eqnarray*}
By the Koebe one-quarter theorem,
we thus see that $\inrad(z;V_z^k) \geq \frac{1}{4}\times\break  M^{-k} s_k
\confrad(z;\D)$. Since $\confrad(z;\D)\geq3/4$ for $z\in
\overline{B(0,1/2)}$, setting $r_k = \frac{3}{16} M^{-k} s_k$ gives
\eqref{eqpVzk}.

A straightforward calculation confirms that these sequences
$\{r_k\}_{k\in\N}$ and $\{R_k\}_{k\in\N}$ satisfy \eqref{eqprR}.
\end{pf}

We define $P_\lp(\Gamma)\subseteq\D$ by
%
%e4.13 #&#
\begin{equation}
\label{eqnpp_definition} P_\lp(\Gamma)\colonequals \bigcap
_{n\geq1} \bigl\{z\in\overline{B(0,1/2)} \dvtx
E_z^{0,n}\mbox{ occurs}\bigr\}.
\end{equation}

Next, we show that elements of $P_\lp(\Gamma)$ are special points of
$\Phi_\lp(\Gamma)$:

%le4.5 #&#
\begin{lemma}
\label{lempptheta_contains}
For $\lp\geq0$, always $P_\lp(\Gamma) \subseteq\Phi_\lp(\Gamma)$.
\end{lemma}
\begin{pf}
It follows from the definition of $E_z^k$ that for $z\in P_\lp$,
the number of loops surrounding $V_z^k$ is $(\lp+o(1))\log s_k^{-1}$ as
$k\to\infty$. Lemma~\ref{lempek_facts} then implies that the number of
loops surrounding $B(0,s_k)$ is also $(\lp+o(1))\log s_k^{-1}$ as
$k\to\infty$.

If $0<\eps<1$, we may choose $k=k(\eps)\geq0$ so that $s_{k+1}\leq
\eps\leq s_k$.
Then
\[
\frac{\Loopcount_z(s_{k})}{\log s_{k}^{-1}} \cdot\frac{\log
s_{k}^{-1}}{\log\eps^{-1}} \leq \frac{\Loopcount_{z}(\eps)}{\log\eps^{-1}}\leq
\frac{\Loopcount_z(s_{k+1})}{\log s_{k+1}^{-1}} \cdot\frac{\log
s_{k+1}^{-1}}{\log\eps^{-1}}.
\]
Observe that $\log s_{k+1} / \log s_k \to1$ as $k\to\infty$.
From this, we see that both the left- hand side and right-hand side
converge to $\lp$ as $\eps\to0$,
so the middle expression also converges to $\lp$ as $\eps\to0$,
which implies $z\in\Phi_\lp(\Gamma)$.
\end{pf}

We use the following lemma, which establishes that the right-hand side of
\eqref{eqnpp_definition} is an intersection of closed sets.

%le4.6 #&#
\begin{lemma} \label{lem:Pnclosed} For each $n\in\N$, the set
$P_{\lp,n}\colonequals\{z\in\overline{B(0,1/2)} \dvtx E_z^{0,n}\mbox{
occurs}\}$ is always closed.
\end{lemma}

\begin{pf}
Suppose that $z$ is in the complement of $P_{\lp,n}$, and let $k$ be
the least value of $j$ such that $E_z^j$ fails to occur. Each of the
two conditions in the definition of $E_z^k$ (see
Lemma~\ref{lempapprox-annulus}) has the property that its failure
implies that $E_w^k$ also does not occur for all $w$ in some
neighborhood of~$z$. (We need the continuity of $\varphi_z^k$ in~$z$,
which may be proved by realizing $\varphi_w^k$ as a composition of
$\varphi_z^k$ with a M\"obius map that takes the disk to itself and the
image of $w$ to $0$.) This shows that the complement of $P_{\lp,n}$ is
open, which in turn implies that $P_{\lp,n}$ is closed.
\end{pf}

%pr4.7 #&#
\begin{proposition}
\label{proppnearind}
Consider a $\CLE_\kappa$ in $\D$. There exists a function $f$ (depending
on $\kappa$ and $\lp$) such that (1) $f(s)=s^{\gamma_\kappa(\lp
)+o(1)}$ as
$s\to0$,
and (2) for all
$z,w \in\overline{B(0,1/2)}$
%
%e4.14 #&#
\begin{equation}
\label{eqpnearind} \P\bigl[E_z^{0,n} \cap E_w^{0,n}
\bigr]f\bigl(\max\bigl(s_n,|z-w|\bigr)\bigr) \leq\P\bigl[E_z^{0,n}
\bigr] \P\bigl[E_w^{0,n}\bigr].
\end{equation}
\end{proposition}

\begin{pf}
Suppose $z,w\in\overline{B(0,1/2)}$. Let $r_k$ and $R_k$ be defined
as in Lem\-ma~\ref{lempek_facts}.
If $|z-w| \leq R_n$, then we bound $\P[E_z^{0,n} \cap E_w^{0,n}] \leq
\P[E_z^{0,n}]$
and, using Lemma~\ref{lempapprox-annulus} and the fact that $R_n=4 s_n$,
\[
\P\bigl[E_z^{0,n}\bigr] \geq\prod
_{k\leq n} t_k^{\gamma_\kappa(\lp
)+o(1)}=s_n^{\gamma_\kappa(\lp
)+o(1)}
= \max\bigl(s_n,|z-w|\bigr)^{\gamma_\kappa(\lp)+o(1)} ,
\]
which implies \eqref{eqpnearind}. Next suppose $|z-w| > R_n$. Letting
\[
u = \min\bigl\{k\in\N\dvtx R_k < |z-w|\bigr\} ,
\]
we have
\[
\label{eqpprodthree} \P\bigl[ E_z^{0,n} \cap E_z^{0,n}
\bigr] = \P\bigl[ E_z^{0,u} \cap E_w^{0,u}
\bigr] \P \bigl[E_z^{u,n}\cap E_w^{u,n}
| E_z^{0,u} \cap E_w^{0,u}\bigr].
\]
By Lemma~\ref{lempek_facts}, $w\notin V_z^u$ and $z
\notin V_w^u$, so we see that $V_z^u$
and $V_w^u$ are disjoint. By the renewal property of
CLE, this implies that conditional on $E_z^{0,u} \cap
E_w^{0,u}$, the events $E_z^k$ and
$E_w^k$ for $k\geq u$ are independent. Thus,
\begin{eqnarray*} \P\bigl[
E_z^{0,n} \cap E_z^{0,n}\bigr] &=& \P
\bigl[ E_z^{0,u} \cap E_w^{0,u}\bigr]
\P \bigl[E_z^{u,n}\bigr] \P\bigl[E_w^{u,n}
\bigr]
\\
&\leq&\P\bigl[E_z^{0,u}\bigr] \P\bigl[E_z^{u,n}
\bigr]\P\bigl[ E_w^{u,n} \bigr]\qquad\mbox{so}
\\
\P\bigl[ E_z^{0,n} \cap E_z^{0,n}
\bigr] \P\bigl[E_w^{0,u}\bigr]&\leq&\P\bigl[
E_z^{0,n}\bigr] \P\bigl[E_w^{0,n}
\bigr].
\end{eqnarray*}
Since
\[
\P\bigl[E_w^{0,u}\bigr] \geq s_u^{\gamma_\kappa(\lp)+o(1)}
= \max \bigl(s_n,|z-w|\bigr)^{\gamma_\kappa(\lp
)+o(1)} ,
\]
\eqref{eqpnearind} follows in this case as well.
\end{pf}

We take $t_k$ as in Section~\ref{secplowerbound} and $s_k$ as in
\eqref{eqnps_k_definition}. We will prove Theorem~\ref{thmpmain} using
Proposition~\ref{proppnearind} and the following general fact about
Hausdorff dimension.
The key ideas in Proposition~\ref{propplowerbound} have appeared in
\cite{frostman,DPRZ,HMP}, but Proposition~\ref{propplowerbound} gives
a cleaner statement that can be used with our construction of nested
closed sets.

%pr4.8 #&#
\begin{proposition}
\label{propplowerbound}
Suppose $P_1\supset P_2\supset P_3\supset\cdots$ is a random nested
sequence of closed sets, and $\{s_n\}_{n\in\N}$ is a sequence of
positive real
numbers converging to 0. Suppose further that $0<a<2$, and
$f(s)=s^{a+o(1)}$ as $s\to0$. If for each $z,w\in\D$ and $n\geq1$ we
have $\P[z\in P_n]>0$ and
%
%e4.15 #&#
\begin{equation}
\label{eqpalmostind} \P[z,w \in P_n] f\bigl(\max\bigl(s_n,|z-w|\bigr)
\bigr) \leq\P[z \in P_n] \P[w \in P_n] ,
\end{equation}
then for any $\alpha<2-a$,
\[
\P\bigl[\dim_\CH(P) \geq\alpha\bigr] >0 \qquad\mbox{where } P\colonequals
\bigcap_{n\geq1} P_n.
\]
\end{proposition}

\newcommand{\randmeas}{\mu}
\begin{pf}
Let $\randmeas_n$ denote the random measure with density with
respect to Lebesgue measure on $\C$ given by
\[
\frac{d\randmeas_n(z)}{dz} = \frac{\one_{z\in P_n\cap\D} }{ \P
[z\in P_n]}.
\]
Then
$\E[\randmeas_n(\D)]=\mathrm{area}(\D)$, and by \eqref{eqpalmostind},
\[
\E\bigl[\randmeas_n(\D)^2\bigr] = \int\!\!\!\int
_{\D\times\D} \frac{\P
[z,w\in
P_n]}{\P[z\in
P_n]\P[w\in P_n]} \,dz \,dw \leq C_1 < \infty
\]
for some constant $C_1$ depending on the function $f$ but not $n$.

For $\alpha\geq0$, the $\alpha$-\textit{energy} of a measure $\mu
$ on
$\C$ is defined by
\[
I_\alpha(\mu)\colonequals\int\!\!\!\int_{\C\times\C}
\frac
{1}{|z-w|^\alpha} \,d\mu(z) \,d\mu(w).
\]
If there exists a nonzero measure with finite $\alpha$-energy
supported on
a set $P\subset\C$, usually called a Frostman measure, then the Hausdorff
dimension of $P$ is at least~$\alpha$ (\cite{MR2118797}, Theorem~4.13). The
expected $\alpha$-energy of $\randmeas_n$ is
\[
\E\bigl[I_\alpha(\randmeas_n)\bigr] = \int\!\!\!\int
_{\D\times\D} \frac{\P[z,w\in P_n]}{\P[z\in
P_n]\P[w\in P_n]} \frac{1}{|z-w|^\alpha} \,dz \,dw ,
\]
and when $\alpha<2-a$, the expected $\alpha$-energy is
bounded by a finite constant $C_2$ depending on $f$ and $\alpha$ but not
$n$.

Since the random variable $\randmeas_n(\D)$ has constant mean and uniformly\break 
bounded variance, it is uniformly bounded away from $0$ with uniformly
positive probability as $n \to\infty$. Also, $\P[I_\alpha(\randmeas_n)
\leq d] \to1$ as $ d\to\infty$ uniformly in $n$. Therefore, we can choose
$b$ and $d$ large enough that the probability of the event
\[
G_n\colonequals\bigl\{b^{-1} \leq\randmeas_n(
\D) \leq b \mbox{ and }I_\alpha(\randmeas_n) \leq d\bigr\}
\]
is bounded away from $0$ uniformly in $n$. It follows that
with positive probability infinitely many $G_n$'s occur.
The set of measures $\randmeas$ satisfying $b^{-1} \leq\randmeas(\D
)\leq b$
and is weakly compact by Prohorov's compactness theorem. Therefore,
on the event that $G_n$ occurs for infinitely many $n$,
there is a sequence of integers $k_1,k_2,\ldots$ for which $\randmeas
_{k_\ell}$ converges
to a finite nonzero measure $\randmeas_\star$ on $\D$.

We claim that $\randmeas_\star$ is supported on $P$.
To show this, we use the
portmanteau theorem, which implies that if $\pi_\ell\to\pi$ weakly
and $U$ is open, then $\pi(U)\leq\liminf_\ell\pi_\ell(U)$.
Since $P_n$ is closed for each $n\in\N$, we have
\[
\randmeas_\star(\C\setminus P_n) \leq\liminf
_{\ell\to\infty}\randmeas_{k_\ell}(\C\setminus P_n) =
0 ,
\]
where the last step follows because $\randmeas_{k_\ell}$ is supported
on $P_{k_\ell}\subset P_n$ for $k_\ell\geq n$. Therefore,
\[
\randmeas_\star(\C\setminus P) = \lim_{n\to\infty}
\randmeas_\star(\C\setminus P_n) = 0 ,
\]
so $\randmeas_\star$ is supported on $P$.

To see that $\randmeas_\star$ has
finite $\alpha$-energy, we again use the portmanteau theorem, which implies
that
\[
\int f\, d\mu\leq\liminf_{\ell\to\infty}\int f \,d\mu_\ell
\]
whenever $f$ is a lower semicontinuous function bounded from below and
$\mu_\ell\to\mu$ weakly. Taking $f(z,w)=|z-w|^{-\alpha}$, $\mu
_\ell=
\randmeas_{k_\ell}(dz) \randmeas_{k_\ell}(dw)$, and $\mu=
\randmeas(dz)\randmeas(dw)$ completes the proof.
\end{pf}

\begin{pf*}{Proof of Theorem~\ref{thmpmain}}
Recall that conformal invariance was proved in
Proposition~\ref{proppPhi-conformal}. We now show that
$\dim_\CH\Phi_\lp(\Gamma) = 2-\gamma_\kappa(\lp)$ almost
surely, when $0\leq
\lp\leq
\lp_{\max}$. (The case $\lp=\lp_{\max}$ uses a separate argument.)
We established the upper bound in
Section~\ref{secpupperbound}, so we just need to prove the lower bound.

Suppose $\lp<\lp_{\max}$. For each connected component $U$ in the
complement of the gasket of $\Gamma$, let $z(U)$ be the
lexicographically smallest rational point in $U$, and let $\varphi_U$
be the Riemann map from $(U,z(U))$ to $(\D,0)$ with positive derivative
at $z(U)$. By Proposition~\ref{propplowerbound}, for any $\eps>0$,
there exists $p(\eps) > 0$ such that
\[
\P\bigl[\dim_\CH\bigl(P_\lp\bigl(\varphi_U(
\Gamma|_U)\bigr)\bigr) \geq2-\gamma_\kappa (\lp)-\eps\bigr]
\geq p(\eps).
\]
By Lemma~\ref{lempptheta_contains} $P_\lp(\varphi_U(\Gamma|_U))
\subset\Phi_\lp(\varphi_U(\Gamma|_U))$, and by conformal invariance
we have $\dim_\CH\Phi_\lp(\varphi_U(\Gamma|_U)) = \dim_\CH
\Phi_\lp(\Gamma|_U)$, which lower bounds $\dim_\CH\Phi_\lp
(\Gamma)$.
Since there are infinitely many components $U$ in the complement of the
gasket, and the $\Gamma|_U$'s are independent, almost surely
$\dim_\CH\Phi_\lp(\Gamma)\geq2- \gamma_\kappa(\lp) -\eps$.
Since $\eps> 0$ was
arbitrary, we conclude that almost surely $\dim_\CH\Phi_\lp(\Gamma)
\geq
2-\gamma_\kappa(\lp)$.

It remains to show that $\Phi_\lp(\Gamma)$ is dense in $D$ almost
surely, for $0\leq\lp< \lp_{\mathrm{max}}$. Let $z$ be a rational
point in $D$, and recall that $U_z^k$ is the complementary
connected component of $D\setminus\Loop_z^k$ which contains $z$.
Almost surely $\Phi_\lp(\Gamma|_{U_z^k})$ has positive Hausdorff
dimension, and in particular is nonempty. Since there are countably
many such pairs $(z,k)$, almost surely
$\Phi_\lp(\Gamma|_{U_z^k})\neq\varnothing$ for each such $z$ and $k$,
and almost surely for each rational point $z$,
$\operatorname{diameter}(U_z^k)\to0$.
\end{pf*}

%
%th4.9 #&#
\begin{theorem} \label{thmpuncountable}
For a $\CLE_\kappa$ $\Gamma$ in a proper simply connected domain
$D$, almost surely $\Phi_{\lp_{\max}}(\Gamma)$ is equinumerous with
$\R$.
Furthermore, almost surely $\Phi_{\lp_{\max}}(\Gamma)$ is dense in $D$.
\end{theorem}

\begin{pf}
As usual, we assume without loss of generality that $D=\D$. We will
describe a random injective map from the set $\{0,1\}^\N$ of binary
sequences to $\D$ such that the image of the map is almost surely a
subset of $\Phi_{\lp_{\max}}(\Gamma)$. The idea of the proof is to find
two disjoint annuli in $\D$ such that $\Gamma$ contains a loop winding
around each annulus, $\Gamma$ has many loops surrounding these annuli,
and for which the nesting of these loops is sufficiently well behaved
for the limiting loop density to make sense.
We then find two further such annuli
inside each of those, and so on. Every binary sequence specifies a path in
the resulting tree of domains, and the intersections of the domains along
distinct paths correspond to distinct points in
$\Phi_{\lp_{\max}}(\Gamma)$.

Let $M>0$ be a large constant as described in
Lemma~\ref{lempapprox-annulus}. For a CLE $\Gamma$ in $\D$, let
$E^{\D}_{0,\eps}(\lp)$ denote the event that there is a loop
contained in
$\overline{B(0,\eps)}\setminus B(0,\eps/M)$ surrounding $B(0,\eps
/M)$ and
such that the index $J$ of the outermost such loop is at least $\lp
\log
\eps^{-1}$. If $(D,z)\neq(\D,0)$, $\Gamma$ is a CLE in $D$, and
$\eps>0$, let $E^D_{z,\eps}(\lp)$ be the event $E^{\D}_{0,\eps
}(\lp)$
occurs for the conformal image of $\Gamma$ under a Riemann map from
$(D,z)$ to $(\D,0)$. If $\{\eps_j\}_{j\in\N}$ is a sequence of positive
real numbers, let
$E^{D,n}_{z}(\lp)=E^{D,n}_{z,\{\eps_j\}_{j=1}^\infty}(\lp)$ denote the
event that $E^D_{z,\eps}(\lp)$ ``occurs $n$ times'' for the first $n$
values of $\eps$ in the sequence. More precisely, we define
$E^{D,n}_{z}(\lp)$ inductively by
$E^{D,1}_{z}(\lp)=E^{D}_{z,\eps_1}(\lp)$ and
%
%e4.16 #&#
\begin{equation}
\label{eqpEdef} E^{D,n}_{z}(\lp) = E^D_{z,\eps_1}(
\lp) \cap E^{U_z^J,n-1}_{z,\{\eps_j\}_{j=2}^\infty}(\lp) ,
\end{equation}
for $n>1$. For the remainder of the proof, we fix the sequence $\eps_j
\colonequals t_j=2^{-j-1}$ and define the events $E^{D,n}_{z}(\lp)$
with respect to this sequence.

For a domain $D$ with $z_0\in D$, let $\varphi$ be a conformal map from
$(\D,0)$ to $(D,z_0)$, and let $F^{D,z_0,n}(\lp)$ denote the event that
there is some point $z\in B(0,1/2)$ for which
$E^{D,n}_{\varphi(z)}(\lp)$ occurs. By Lemma~\ref{lem:Pnclosed} and
Propositions \ref{proppnearind} and \ref{propplowerbound}, we see
that there is some $p>0$ [depending on $\kappa\in(8/3,8)$ and
$\lp<\lp_{\max}$], such that $\P[F^{D,z_0,n}(\lp)]\geq p$ for
all $n$.

For each $k\in\N$, we choose $\lp_k\in(\lp_{\mathrm{typical}},\lp
_{\max})$ so that
$\gamma_\kappa(\lp_k)=2-2^{-k-1}$. For each $k\in\N$ and $\ell\in
\N$,
we define $q_{k,\ell}=2^{-2k-\ell}$.

Suppose $z\in\ol{B(0,1/2)}$ and $0<r<1/2$ and $0<u<r/M$. For $n\in\N$
and $\lp<\lp_{\max}$, we say that the annulus $B(z,r)\setminus B(z,u)$
is $(n,\lp)$-\textit{good} if (i) there exists a loop contained in the
annulus and surrounding $z$ (say $\Loop_z^j$ is the outermost such
loop) and (ii) the event $F^{U_z^j,z,n}(\lp)$ occurs.

For $q,r>0$, define $u(q,r)=(q/C)^{1/\alpha} r^{1+2/\alpha}$, where $C$
and $\alpha$ are chosen so that every annulus $B(z,r)\setminus B(z,u)$
contained in $D$ contains a loop surrounding $z$ with probability at
least $1-C (u/r)^{\alpha}$ (see Lemma~\ref{lempannulus-loop}). For
$0<r<1/2$, let $S_r$ be a set of $\frac{1}{100r^2}$ disjoint disks of
radius $r$ in $B(0,1/2)$. By our choice of $u(q,r)$, the event $G$ that
all the disks $B(z,r)$ in $S_{r}$ contain a CLE loop surrounding
$B(z,u)$ has probability at least $1-q$. We choose $r_{k,\ell}>0$
small enough so that for all $n\in\N$, with probability at least
$1-2^{1-2k-\ell}$ there are two disks $B(z,r_{k,\ell})$ in
$S_{r_{k,\ell}}$ such that $B(z,r_{k,\ell})\setminus
B(z,u(q_{k,\ell},r_{k,\ell}))$ is an $(n_k,\lp_k)$-good annulus.
This is
possible because on the event $G$, the disks in $S_{r_{k,\ell}}$ give
us $\frac{1}{100r_{k,\ell}^2}$ independent trials to obtain a good
annulus, and each has success probability at least $p$. Abbreviate
$u_{k,\ell} = u(q_{k,\ell},r_{k,\ell})$. Finally, we define a sequence
$(n_k)_{k\in\N}$ growing sufficiently fast that
%
%e4.17 #&#
\begin{equation}
\label{eqpnfast} \lim_{k\to0}\frac{\sum_{j=1}^k \log u_{j+1,1}^{-1}}{\sum_{j=1}^k
\log s_{n_j}^{-1}} = 0.
\end{equation}

Now suppose that $\Gamma$ is a CLE in the unit disk. Define
\[
A=A(\D,0,r,u,q,n,\lp)
\]
to be the event that there are at least
two disks $B(z,r)$ and $B(w,r)$ in $S_{r}$ such that $B(z,r)\setminus
B(z,u)$ and $B(w,r)\setminus B(w,u)$ are both $(n,\lp)$-good. If
$(D,z)\neq(\D,0)$ and $\Gamma$ is a CLE in $D$, define
$A=A(D,z,r,u,q,n,\lp)$ to be the event that
$A(\D,0,r,u,q,n,\lp)$ occurs for the conformal image of $\Gamma$
under a Riemann map from $(D,z)$ to $(\D,0)$. Abbreviate
$A(D,z,r_{k,\ell},u_{k,\ell},q_{k,\ell},n_k,\lp_k)$ as $A_{k,\ell}(D,z)$.

We define a random map $b\mapsto D_b$ from the set of terminating
binary sequences to the set of subdomains of $\D$ as follows. If the
event $A_{1,1}(\D,0)$ occurs, we set $\ell(\D)=1$ and define
$D_0=\varphi_z^{-1}(U_{\acute{z}}^{I(\acute{z})})$ and
$D_1=\varphi_w^{-1}(U_{\acute{w}}^{I(\acute{w})})$, where $z$ and $w$
are the centers of two $(n_1,\lp_1)$-good annuli, $\varphi_z$
(resp., $\varphi_w$) is a Riemann map from $(\D,z)$ [resp., $(D,w)$] to
$(\D,0)$, $\acute{z}\in U_z^{J_{z,r_{1,1}}^\subset}$ and $\acute{w}
\in
U_w^{J_{w,r_{1,1}}^\subset}$ are points for which
$E^{U_z^{J_{z,r_{1,1}}^\subset},n_1}_{\acute{z}}(\lp_1)$ and
$E^{U_w^{J_{w,r_{1,1}}^\subset},n_1}_{\acute{w}}(\lp_1)$ occur, and
$I(\acute{z})$ [resp., $I(\acute{w})$] is the index of the $n$th loop
encountered in the definition of $E^{D,n}_{\acute{z}}(\lp_1)$
[resp.,
$E^{D,n}_{\acute{w}}(\lp_1)$] [in other words, the first such loop is
denoted $J$ in \eqref{eqpEdef}, the second such loop is the first one
contained in the preimage of $B(0,\eps_2)$ under a Riemann map from
$(U_z^J,z)$ to $(D,0)$, and so on]. If $A$ does not occur, then we
choose a disk $B(z,r_{1,1})$ in $S_{r_{1,1}}$ and consider whether the
event\vspace*{-2pt} $A_{1,2}(U{}^{J_{z,r_{1,1}}^\subset}_{z})$ occurs. If it does, then
we set $\ell(\D) = 2$ and define $D_0$ and $D_1$ to be the conformal
preimages of $U_z^{J_{z,r_{1,2}}^\subset}$ and
$U_w^{J_{w,r_{1,2}}^\subset}$, respectively, where again $z$ and $w$
are centers of two $(n_1,\lp_1)$-good annuli. Continuing inductively in
this way, we define $\ell(\D) \in\N$ and $D_0$ and $D_1$ [note that
$\ell(\D)<\infty$ almost surely by the Borel--Cantelli lemma since
$\sum_{\ell}2^{1-2k-\ell}<\infty$]. Repeating this procedure in $D_0$
and $D_1$ beginning with $k=2$ and $\ell= 1$, we obtain
$D_{i,j}\subset D_i$ for $i,j\in\{0,1\}\times\{0,1\}$. Again
continuing inductively, we obtain a map $b\mapsto D_b$ with the
property that $D_{b} \subset D_{b'}$ whenever $b'$ is a prefix of $b$.

If $b\in\{0,1\}^\N$, we define $z_b =
\bigcap_{b'\mathrm{is\ a\ prefix\ of\ }b} D_{b'}$. Since $\sum_{k}
2^k2^{-2k-\ell} < \infty$, with probability 1 at most finitely many of
the domains $D_b$ have $\ell(D_b) >0$. It follows from this observation
and \eqref{eqpnfast} that
\[
\liminf_{t\to0} \TLoopcount_{z_b}(t) \geq
\lp_{\max}.
\]
But by Proposition~\ref{proppupper_bound}, almost
surely every point $z$ in $\D$ satisfies
\[
\limsup_{t\to0} \TLoopcount_z(t) \leq
\lp_{\max}.
\]
Therefore, $z_b \in
\Phi_{\lp_{\max}}(\Gamma)$.

Since the set of binary sequences is equinumerous with $\R$, this
concludes the proof that $\Phi_{\lp_{\max}}(\Gamma)$ is equinumerous
with $\R$. The proof that $\Phi_{\lp_{\max}}(\Gamma)$ is dense now
follows using the argument for density in Theorem~\ref{thmpmain}.
\end{pf}

%s5 #&#
\section{Weighted loops and Gaussian free field extremes}
\label{secpweighted_loops}

The main result of this section is
Theorem~\ref{thmparbitrary_distribution}, which generalizes
Theorem~\ref{thmpmain} and highlights the connection between extreme
loop counts
and the extremes of the Gaussian free field \cite{HMP}. Let $\Gamma$
be a
$\CLE_\kappa$, and fix a probability measure $\mu$ on $\R$.
Conditional on $\Gamma$, let $(\xi_\Loop)_{\Loop\in\Gamma}$ be an
i.i.d.
collection of $\mu$-distributed random variables indexed by $\Gamma$. For
$z \in D$ and $\eps> 0$, we let $\Gamma_z(\eps)$ \label
{notpGamma_z} be
the set of loops in $\Gamma$ which surround $B(z,\eps)$ and define
\[
\SLoopcount_z(\eps)=\sum_{\Loop\in
\Gamma_z(\eps)}
\xi_\Loop\quad\mbox{and}\quad \TLoopsum_z(\eps) =
\frac
{\SLoopcount_z(\eps)}{\log(1/\eps)}.
\]
For a $\CLE_\kappa$ $\Gamma$ on a domain $D$ and $\gp\in\R$, we
define $\Phi^\mu_\gp(\Gamma) \subset
D$ by
\[
\Phi^\mu_\gp(\Gamma) \colonequals \Bigl\{z\in D \dvtx
\lim_{\eps\to0} \TLoopsum_z(\eps)=\gp \Bigr\}.
\]

To study the Hausdorff dimension of $\Phi^\mu_\gp(\Gamma)$, where
$\Gamma$ is a $\CLE_\kappa$ on
$D$, we introduce for each $(\gp,\lp) \in\R\times[0,\infty)$ the set
%
%e5.1 #&#
\begin{equation}
\label{eqnpPhi-k-mu-a-nu} \Phi^\mu_{\gp,\lp}(\Gamma) \colonequals \Bigl
\{z\in D \dvtx \lim_{\eps\to0} \TLoopsum_z(\eps)=\gp
\mbox{ and } \lim_{\eps\to0}\TLoopcount_z(\eps)=\lp \Bigr
\}.
\end{equation}
Let $\Lambda_\mu^\star$ be the Fenchel--Legendre transform of $\mu$ and
let $\Lambda_\kappa^\star$ be the Fenchel--Legendre transform of the
$\log$
conformal radius distribution \eqref{eqnpLambdaR}. We define
%
%e5.2 #&#
\begin{equation}
\label{eqnpg_alpha_nu_def} \gamma_\kappa(\gp,\lp) = \cases{\displaystyle \lp
\Lambda_\mu^\star \biggl(\frac{\gp}{\lp} \biggr) + \lp
\Lambda _\kappa^\star \biggl(\frac{1}{\lp} \biggr), &\quad $
\lp>0,$\vspace *{2pt}
\cr
\displaystyle\lim_{\lp'\searrow0} \gamma_\kappa\bigl(
\gp,\lp'\bigr), &\quad $\lp=0\mbox{ and }\gp\neq 0,$ \vspace*{2pt}
\cr
\displaystyle\lim
_{\lp'\searrow0} \gamma_\kappa\bigl(\lp'\bigr)=1 -
\frac{2}{\kappa} - \frac
{3\kappa}{32}, & \quad $\lp=0\mbox{ and }\gp= 0 ,$}
\end{equation}
where the limits exist by the convexity of $\Lambda_\kappa^\star$ and
$\Lambda_\mu^\star$
[Proposition~\ref{propptransform_properties}(i)]. Note
that $\gamma_\kappa(\gp,\lp)$ may be infinite for some $(\gp,\lp
)$ pairs.
Note also
that the second and third limit expressions for $\gp=0$, $\lp=0$ agree
except when $\Lambda_\mu^\star(0)=\infty$, because $\lim_{\lp'
\to0} \lp'
\Lambda_\mu^\star(0/\lp') = 0$ whenever $\Lambda_\mu^\star
(0)<\infty$.

%th5.1 #&#
\begin{theorem}
\label{thmparbitrary_distribution_joint}
Suppose $\lp\geq0$, $\gp\in\R$,
$\Phi^\mu_{\gp,\lp}(\CLE_\kappa)$ is given by \eqref{eqnpPhi-k-mu-a-nu},
and $\gamma_\kappa(\gp,\lp)$ is given by
\eqref{eqnpg_alpha_nu_def}.
If $\gamma_\kappa(\gp,\lp) \leq2$, then almost surely,
%
%e5.3 #&#
\begin{equation}
\label{eqnpweighted_dim} \dim_\CH\Phi^\mu_{\gp,\lp}(
\CLE_\kappa) = 2-\gamma_\kappa(\gp ,\lp).
\end{equation}
If $\gamma_\kappa(\gp,\lp) > 2$, then almost surely $\Phi^\mu
_{\gp,\lp
}(\CLE_\kappa) = \varnothing$.
\end{theorem}
\begin{pf}
Suppose that $\Gamma\sim\CLE_\kappa$ in a proper simply connected
domain $D \subset\C$.
If $\gp= \lp= 0$, then $\Phi^\mu_{\gp,\lp}(\Gamma)$ contains the gasket
of $\Gamma$, which implies $\dim_\CH\Phi^\mu_{\gp,\lp}(\Gamma)
\geq
2-\gamma_\kappa(0,0)$ \cite{NW,MSW}. Furthermore, $\Phi^\mu_{\gp
,\lp}(\Gamma)\subset
\Phi_{0}(\Gamma)$, which implies by Theorem~\ref{thmpmain} that
$\dim_\CH
\Phi^\mu_{\gp,\lp}(\Gamma) \leq2-\gamma_\kappa(0,0)$. Therefore,
\eqref{eqnpweighted_dim} holds in the case $\gp= \lp= 0$.

Suppose that $(\gp,\lp)\neq(0,0)$, and assume $\gamma_\kappa(\gp
,\lp)\leq
2$. For the upper bound in~\eqref{eqnpweighted_dim}, we follow the proof
of Proposition~\ref{proppupper_bound}. As before, we restrict our
attention without loss of generality to the case that $D = \D$ and the set
$\Phi^\mu_{\gp,\lp}(\Gamma) \cap B(0,1/2)$.

For the remainder of the proof, we interpret the expression
$0\Lambda^\star(\gp/0)$ to mean $\lim_{\lp\to
0}\lp\Lambda^\star(\gp/\lp)$ for $\Lambda^\star\in
\{\Lambda_\mu^\star,\Lambda_\kappa^\star\}$ and $\gp\in\R$. Fix
$\eps>0$. We claim that for $\delta>0$ sufficiently small,
%
%e5.4 #&#
%e5.5 #&#
\begin{eqnarray}
\label{eqpf_cont_kappa} \inf_{\lp' \in(\lp-\delta,\lp+\delta) \cap[0,\infty)} \lp'
\Lambda_\kappa^\star \biggl(\frac{1}{\lp'} \biggr) &\geq&\lp
\Lambda_{\kappa}^\star \biggl(\frac{1}{\lp} \biggr)-
\frac{\eps}{8}\quad \mbox{and}
\\
\label{eqpf_cont_mu} \mathop{\inf_{\lp' \in(\lp-\delta,\lp+\delta)\cap[0,\infty),
}}_{ \gp'\in(\gp-\delta,\gp+\delta) }
\lp'\Lambda_\mu^\star \biggl(
\frac{\gp'}{\lp'} \biggr) &\geq& 3 \wedge \biggl(\lp \Lambda_\mu^\star
\biggl(\frac{\gp}{\lp} \biggr) - \frac{\eps
}{8} \biggr).
\end{eqnarray}
[We include the minimum with $3$ on the right-hand side of \eqref
{eqpf_cont_mu} to handle the case that $\lp\Lambda_\mu^\star(\gp
/\lp) = \infty$. The particular choice of $3$ was arbitrary; any
value strictly larger than $2$ would suffice.]

The continuity of $\lp
\Lambda_\kappa^\star(1/\lp)$ on $[0,\infty)$
(Proposition~\ref{propplimit_at_zero}) implies \eqref{eqpf_cont_kappa}.

For \eqref{eqpf_cont_mu}, we
consider three cases:
\begin{longlist}[(iii)]
\item[(i)] If $\lp>0$, then \eqref{eqpf_cont_mu} follows from the lower
semi-continuity of $\Lambda^\star_\mu$ (see the definitions in the
beginning of \cite{DZ}, Section~1.2, and \cite{DZ}, Lemma~2.2.5).
\item[(ii)] If $\lp= 0$ (so that $\gp\neq0$) and $\lim_{x\to0}x\Lambda
_\mu^\star(1/x)<\infty$, we write
%
%e5.6 #&#
\begin{equation}
\label{eqpproduct} \lp' \Lambda_\mu^\star
\biggl(\frac{\gp'}{\lp'} \biggr) = \gp'\cdot \biggl(
\frac{\lp'}{\gp'} \Lambda_\mu^\star \biggl(
\frac{\gp'}{\lp'} \biggr) \biggr).
\end{equation}
Assume that $\gp' > 0$; the case that $\gp' < 0$ is symmetric. If
$\delta\in(0,\gp)$, then $\gp' \in(\gp- \delta,\gp+\delta)$
implies that
$\gp'$ is bounded away from 0. Therefore, \eqref{eqpproduct} and the
lower semi-continuity of $\Lambda^\star_\mu$ imply that for all
$\eta
>0$, there exists $\delta>0$ such that
\[
\frac{\lp'}{\gp'}\Lambda_\mu^* \biggl(\frac{\gp'}{\lp'} \biggr)
\geq \lim_{x\to0}x\Lambda_\mu^*(1/x) - \eta,
\]
whenever $0<\lp'<\delta$ and $\gp'\in
(\gp-\delta,\gp+\delta)$. Since $\gp'>\gp- \delta$, we can choose
$\eta>0$ and then $\delta>0$ sufficiently small that
\eqref{eqpf_cont_mu} holds.
\item[(iii)] If $\lp= 0$ (so that $\gp\neq0$) and $\lim_{x\to0}x\Lambda
_\mu^\star(1/x)=\infty$, then the lower semicontinuity of $\Lambda
_\mu^\star$ implies that there exists $\delta> 0$ such that \eqref
{eqpf_cont_mu} holds with $3$ on the right-hand side.
\end{longlist}

We choose $\delta> 0$ so that \eqref{eqpf_cont_kappa} and \eqref
{eqpf_cont_mu} hold, and we replace the definition
\eqref{eqnpUdef} of $\CU^{r,\lp+}$ with
\[
\CU^{r,\lp,\gp}\colonequals \bigl\{U \in\CD^r \dvtx \bigl|
\TLoopcount_{z(U)}(r) -\lp \bigr| \leq\delta\mbox{ and } \bigl|
\TLoopsum_{z(U)}(r)- \gp \bigr| \leq\delta \bigr\} ,
\]
where $\CD^{r}$ is defined as in Section~\ref{secpupperbound} in the
proof of Proposition~\ref{proppupper_bound}. As
in \eqref{eqnpcover_def}, $\CC^{m,\lp,\gp} = \bigcup_{n \geq m}
\CU^{\exp(-n),\lp,\gp}$
is a cover of $\Phi^\mu_{\gp,\lp}(\Gamma) \cap B(0,1/2)$ for all
$m\in\N$. Suppose that $\gamma_\kappa(\gp,\lp) \leq2$. Using
Lemma~\ref{lempasymp_probability_disk} and
\hyperref[thmpcramer]{Cram\'er's theorem}, we see that for
sufficiently large $n$,
%
%e5.7 #&#
\begin{eqnarray}\label{eqnpweight_upper_bound}
\nonumber
\P\bigl[U \in\CU^{\exp(-n),\lp,\gp}\bigr] &\leq&\P \bigl[\bigl |
\TLoopsum_{z(U)}\bigl(e^{-n}\bigr)-\gp \bigr| \leq\delta |\bigl |
\TLoopcount_{z(U)}\bigl(e^{-n}\bigr)-\lp \bigr| \leq\delta \bigr]\\
&&{}\times  \P \bigl[\bigl |\TLoopcount_{z(U)}\bigl(e^{-n}\bigr) -
\lp\bigr | \leq\delta \bigr]
\\
&\leq& e^{-(\gamma_\kappa(\gp,\lp)-\eps/2)n}.\nonumber
\end{eqnarray}
If $\gamma_\kappa(\gp,\lp)>2$, then the same analysis shows that
$\P[U \in
\CU^{\exp(-n),\lp,\gp}] \leq e^{-cn}$ for some $c>2$. The rest of the
argument now follows the proof of Proposition~\ref{proppupper_bound}.

For the lower bound, we may assume $\gamma_\kappa(\gp,\lp)\leq2$, which
implies that
$\lp\Lambda_\mu^\star(\gp/\lp)$ is finite. We consider the events denoted
by $E_z^k$ in the discussion following Lemma~\ref{lempapprox-annulus},
which we now denote by $E_z^k(1)$. We also define events on which we can
control the sums associated with the loops in each annulus. More
precisely, suppose that $(\delta_k)_{k \in\N}$ is a sequence of positive
real numbers with $\delta_k \to0$ as $k \to\infty$. We define
\[
E_z^k(2) = \bigl\{\CS_0
\bigl(t_k; \wt{\Gamma}_z^k\bigr) \in \bigl( (
\gp- \delta_k )\log t_k^{-1}, (\gp+
\delta_k )\log t_k^{-1} \bigr) \bigr\}.
\]
[Recall the definition of $\wt{\Gamma}_z^k$ from
Section~\ref{secplowerbound} and that $\SLoopcount_0(t_k;\wt{\Gamma}_z^k)$
represents the weighted loop count with respect to $\wt{\Gamma}_z^k$, where
we define $\xi_\Loop$ for $\Loop\in\wt{\Gamma}_z^k$ to be equal
to the
weight of the conformal preimage of $\Loop$ in $\Gamma$.] We define the
events $\acute{E}_z^k = E_z^k(1)\cap E_z^k(2)$ and $\acute
{E}_z^{k_1,k_2} =
\bigcap_{k=k_1}^{k_2} \acute{E}_z^k$ as before. Similar to
\eqref{eqnpweight_upper_bound}, we have by
\hyperref[thmpcramer]{Cram\'er's theorem}
\[
\P \bigl[E_z^k(2) | E_z^k(1)
\bigr] = t_k^{\lp
\Lambda_\mu^\star(\gp/\lp)+o(1)} ,
\]
provided $\delta_k \to0$ slowly enough. We multiply both sides by
$\P[E_z^k(1)] = \break  t_k^{\lp\Lambda_\kappa^\star(1/\lp)+o(1)}$ and get
\[
\P\bigl[\acute{E}_z^k\bigr] = t_k^{\gamma_\kappa(\gp,\lp)+o(1)}\qquad
\mbox{as } k \to\infty.
\]
Thus, Proposition~\ref{proppnearind} and its proof carry over with
$\gamma_\kappa(\lp)$ replaced by $\gamma_\kappa(\gp,\lp)$.\vspace*{1pt}

It remains to verify that $\acute{P}(\gp,\lp;\Gamma) \subset
\Phi^\mu_{\gp,\lp}(\Gamma)$, where $\acute{P}(\gp,\lp;\Gamma)$
is defined
to be the set of points $z$ for which $\acute{E}_z^{1,n}$ occurs for all
$n$. We see that $\lim_{\eps\to0} \TLoopcount_z(\eps) = \lp$ for the
reasons explained in the proof of Lemma~\ref{lempptheta_contains}. Moreover, $\lim_{\eps\to0} \TLoopsum
_z(\eps)
= \gp$ for analogous reasons. By Proposition~\ref{propplowerbound}, this
completes the proof.
\end{pf}

In Theorem~\ref{thmparbitrary_distribution}, we show that $\dim_\CH
\Phi
^\mu_\gp(\CLE_\kappa)$ is almost surely equal to the
maximum of the expression given in Theorem~\ref
{thmparbitrary_distribution_joint}
as $\lp$ is allowed to vary. In Theorem~\ref{thmpunique-minimizer}
we show that,
with the exception of some degenerate cases,
there is a unique value of $\lp$ at which this maximum is
achieved.

%th5.2 #&#
\begin{theorem}
\label{thmpunique-minimizer}
Let $\gp\in\R$ and let $\mu$ be a probability measure on $\R$.
\begin{longlist}[(iii)]
\item[(i)]%\label{itemppsi_zero}
If $\gp=0$, then $\lp\mapsto\gamma
_\kappa(\gp
,\lp)$ has
a unique nonnegative minimizer $\lp_0$.
\item[(ii)]%\label{itempunique_minimizer}
If $\gp>0$ and $\mu((0,\infty
))>0$ or
if $\gp<0$ and $\mu((-\infty,0))>0$, then $\lp\mapsto\gamma
_\kappa(\gp,\lp)$
has a unique minimizer $\lp_0$. Furthermore, $\lp_0>0$.
\item[(iii)]%\label{itempall-nu-bad}
If $\gp>0$ and $\mu((0,\infty))=0$ or $\gp<0$ and $\mu((-\infty
,0))=0$, then for all $\lp\in[0,\infty)$ we have $\gamma_\kappa
(\gp,\lp
)=\infty$.
In this case we set $\lp_0=0$.
\end{longlist}
\end{theorem}

\begin{pf}
For part (i), note that when $\gp= 0$, the expression
we seek to minimize is $\lp\Lambda_\mu^\star(0) + \lp
\Lambda_\kappa^\star(1/\lp)$. If $\Lambda_\mu^\star(0)<+\infty
$, then
this expression has a unique positive minimizer because its derivative
with respect to $\lp$ differs from that of $\lp
\Lambda_\kappa^\star(1/\lp)$ by the constant $\Lambda_\mu^\star
(0)$ and,
therefore, varies strictly monotonically from $-\infty$ to $+\infty$. If
$\Lambda_\mu^\star(0)=+\infty$, then $\lp= 0$ is the unique minimizer.

For part (iii), observe by
\hyperref[thmpcramer]{Cram\'er's theorem} that
$\Lambda_\mu^\star(x)=\infty$ when $x$ and $\gp$ have the same sign,
so $\gamma_\kappa(\gp,\lp)=\infty$.

For part (ii), we may assume without loss of
generality that $\gp> 0$ and $\mu((0,\infty))>0$. Define $a =
\essinf
X$ and $b = \esssup X$ for a $\mu$-distributed random variable $X$, so
that $-\infty\leq a \leq b \leq+\infty$. Since $\mu((0,\infty
))>0$, we
have $b>0$ by
Proposition~\ref{propptransform_properties}(v).

We make
some observations about the functions $f_\mu\dvtx (0,\infty) \to[0,\infty]$
and $f_\kappa\dvtx (0,\infty) \to[0,\infty]$ defined by
\[
f_\mu(\lp) \colonequals\lp\Lambda_\mu^\star
\biggl(\frac{\gp
}{\lp} \biggr) \quad\mbox{and}\quad f_\kappa(\lp)\colonequals\lp
\Lambda_\kappa^\star \biggl(\frac{1}{\lp} \biggr).
\]
First, they inherit convexity
from $\Lambda_\mu^\star$ and $\Lambda_\kappa^\star$ by Lemma~\ref
{lempconvex}. Note that the sum $f(\lp)\colonequals
f_\mu(\lp) + f_\kappa(\lp)$ is also convex.

By Proposition~\ref{propptransform_properties}(viii),
$\Lambda_\mu^\star$ is continuously differentiable on $(a,b)$.
The chain rule gives
\[
f_\mu'(\lp) = -\frac{\gp}{\lp}\bigl(
\Lambda_\mu^\star\bigr)' \biggl(
\frac
{\gp}{\lp} \biggr) + \Lambda_\mu^\star \biggl(
\frac{\gp}{\lp} \biggr).
\]
If $a>-\infty$, then
Proposition~\ref{propptransform_properties}(ix)
implies $(\Lambda_{\mu}^\star)'(x)\to-\infty$ as $x \searrow a$.
Similarly, if $b<\infty$, Proposition~\ref
{propptransform_properties}(x)
implies $(\Lambda_{\mu}^\star)'(x)\to+\infty$ as $x \nearrow b$.
In other words,
%
%e5.8 #&#
%e5.9 #&#
\begin{eqnarray}
\lim_{\lp\nearrow\gp/a} f'_\mu(\lp) &=& +\infty\qquad
\mbox{if } a > 0\quad \mbox{ and} \label{eqnpinfinitederivative_a}
\\
\lim_{\lp\searrow\gp/b} f'_\mu(\lp) &=& -\infty\qquad
\mbox{if } b<\infty. \label{eqnpinfinitederivative_b}
\end{eqnarray}

Recall from Proposition~\ref{proppparameterization} that (note
$f_\kappa= \gamma_\kappa$)
\[
\bigl\{(\lp,f_\kappa(\lp)\dvtx 0<\lp<\infty\bigr\} =
\biggl\{ \biggl(
\frac{1}{\Lambda_\kappa'(\mgfparam)}, \mgfparam- \frac{\Lambda_\kappa(\mgfparam)}{\Lambda_\kappa'(\mgfparam)} \biggr) \dvtx -\infty<
\mgfparam< 1 - \frac{2}{\kappa} - \frac{3\kappa}{32} \biggr\}.
\]
Suppose $-\infty< \lambda_0 < 1-2/\kappa- 3\kappa/32$. If $\lp=
1/\Lambda_\kappa'(\lambda_0)$, then
\[
f_\kappa'(\lp) = \biggl(\frac{d}{d\lambda}\bigg|_{\lambda= \lambda
_0} \bigl[\lambda-
\Lambda_\kappa(\lambda)/\Lambda_\kappa'(\lambda) \bigr]\biggr)\Big/\biggl(\frac
{d}{d\lambda}\bigg|_{\lambda= \lambda_0} \bigl[1/\Lambda_\kappa
'(\lambda) \bigr]\biggr) = -
\Lambda_\kappa(\lambda_0).
\]
When we take $\lambda_0 \to-\infty$ (which corresponds to taking
$\lp
\to+\infty$) and $\lambda_0 \to1-2/\kappa- 3\kappa/32$ (which
corresponds to taking $\lp\to0$), respectively, in the explicit
formula \eqref{eqnpLambdaR} for $\Lambda_\kappa$, we obtain
%
%e5.10 #&#
%e5.11 #&#
\begin{eqnarray}
\label{eqnpinfinite_slope_a} \lim_{\lp\searrow0}f'_\kappa(
\lp) &=& -\infty\quad \mbox{and}
\\
\lim_{\lp\nearrow+\infty}f_\kappa'(\lp) &=&+\infty.
\label
{eqnpinfinite_slope_b}
\end{eqnarray}

We complete the proof of (ii) by treating five
cases separately. For each of the cases (i)--(ii) and (iv)--(v), we argue
that $f'(\lp)$ ranges from $-\infty$ to $+\infty$ for
$\lp\in(\gp/b,\gp/ \max(0,a))$ [if $a < 0$ so that $\max(0,a) = 0$
then we interpret $\gp/0 = +\infty$]. Upon showing this,
continuous differentiability of $f$
[Proposition~\ref{propptransform_properties}(viii)] guarantees
by the intermediate value theorem that the equation $f'(\lp)=0$ has a
solution. The convexity of $f_\mu$ and strict convexity of $f_\kappa$
(Proposition~\ref{proppstrictly_convex}) imply that the solution is
unique. Case (iii) uses a separate (easy) argument.
\begin{longlist}[(iii)]
\item[(i)]$a \leq0 < b<\infty$. Note that $f'_\mu(x) \to-\infty$ as
$x\searrow\gp/b$ and $f'_\kappa(x)\to+\infty$ as $x\to+\infty$. Since
$f'_\kappa(x) \not\to\infty$ as $x\searrow\gp/b$ and $f'_\mu
(x)\not\to
-\infty$ as $x\to+\infty$, we conclude that $f'((\gp/b,+\infty))=
(-\infty,+\infty)$.
\item[(ii)]$a\leq0 < b = \infty$. We have $f'((0,+\infty))=(-\infty
,+\infty)$
since $f'_\kappa(x)$ goes to $-\infty$ as $x\searrow0$ and to
$+\infty$ as $x\to+\infty$.
\item[(iii)]$0<a=b<\infty$.
Since $a=b$, $\Lambda_\mu^\star(x) = +\infty$ for all $x\neq b$, so
$\lp= \gp/b$ is the unique minimizer of
$\lp\mapsto\gamma_\kappa(\gp,\lp)$.
\item[(iv)]$0< a < b < \infty$. We have $f'((\gp/b,\gp/a))=(-\infty
,+\infty)$
since $f'_\mu(x)$ goes to $-\infty$ as $x\searrow\gp/b$ and to
$+\infty$ as $x\nearrow\gp/a$.
\item[(v)]$0 < a < b = \infty$. We have $f'((0,\gp/a))=(-\infty,+\infty
)$ since
$f_\kappa'(x)$ goes to $-\infty$ as $x\searrow0$ and $f'_\mu(x)$ goes
to $+\infty$ as $x\nearrow\gp/a$. \quad\qed%\qedhere
\end{longlist}
\noqed\end{pf}

%th5.3 #&#
\begin{theorem}
\label{thmparbitrary_distribution}
Let $\gp\in\R$ and let $\mu$ be a probability measure on $\R$.
Let $\lp_0=\lp_0(\gp)$ be the minimizer of $\lp\mapsto\gamma
_\kappa(\gp,\lp)$
from Theorem~\ref{thmpunique-minimizer}.
If $\gamma_\kappa(\gp,\break \lp_0(\gp)) \leq2$, then almost surely
%
%e5.12 #&#
\begin{equation}
\dim_\CH\Phi^\mu_\gp(\CLE_\kappa) =
2-\gamma_\kappa\bigl(\gp,\lp _0(\gp)\bigr).
\end{equation}
If $\gamma_\kappa(\gp,\lp_0(\gp))>2$, then
$\Phi^\mu_\gp(\CLE_\kappa)=\varnothing$ almost surely.
\end{theorem}

\begin{pf}
The lower bound is immediate from
Theorem~\ref{thmparbitrary_distribution_joint}, since
\[
\Phi_\gp^\mu(\Gamma)\supset\Phi_{\gp,\lp_0(\gp)}^\mu(
\Gamma) ,
\]
where $\Gamma$ is a $\CLE_\kappa$. For the upper
bound, we follow the approach in the proof of
Proposition~\ref{proppupper_bound}. It suffices to consider the case
where the domain is the unit disk~$\D$, and without loss of generality we
may consider the set $\Phi^\mu_\gp(\Gamma) \cap
B(0,1/2)$. Observe that if $\gp=0$, then
%
%e5.13 #&#
\begin{equation}
\label{eqpalpha_equals_0} \gamma_\kappa(0,\lp) = \lp\Lambda_\mu^\star(0)
+ \lp \Lambda_\kappa^\star(1/\lp).
\end{equation}
If $\Lambda_\mu^\star(0)=\infty$, then the first term in
\eqref{eqpalpha_equals_0} is infinite unless $\lp= 0$. It follows that
$\lp_0(0)=0$ in this case. If $\Lambda_\mu^\star(0)<\infty$, then the
derivative of the first term with respect to $\lp$ is a nonnegative
constant $\Lambda_\mu^\star(0)$, while the derivative of the second term
is a strictly increasing function going from $-\infty$ to $\infty$ as
$\lp$ goes from 0 to $\infty$. It follows that
$\Lambda_\mu^\star(0)<\infty$ implies $\lp_0(0)>0$. We first
handle the
case $\Lambda_\mu^\star(0)<\infty$.

Let $c_\mu(\gp)=\gamma_\kappa(\gp,\lp_0(\gp))$. Since
$\lp\Lambda_\kappa^\star(1/\lp)$ and $\lp\Lambda_\mu^\star(\gp
/\lp)$
are convex and lower semicontinuous, we may define $\lp_1$ and
$\lp_2$ so that $\lp\Lambda_\kappa^\star(1/\lp) \leq
c_\mu(\gp)$ if and only if $0\leq\lp_1 \leq\lp\leq\lp_2<\infty$.
Observe that $[\lp_1,\lp_2]$ is nonempty
since it contains $\lp_0(\gp)$. We also define $\lp_1' \colonequals
\inf
\{\lp\geq\lp_1 \dvtx  \lp\Lambda_\mu^\star(\gp/\lp) \leq c_\mu(\gp
)\}$
and $\lp_2' \colonequals\sup\{\lp\leq\lp_2 \dvtx \lp
\Lambda_\mu^\star(\gp/\lp) \leq c_\mu(\gp)\}$.

We claim that
%
%e5.14 #&#
\begin{equation}\label{eqnpunif_cont}
\begin{tabular}{p{300pt}@{}}
$\forall\eps>0, \exists\delta>0$ so that $\forall\bigl(
\gp',\lp\bigr) \in[\gp- \delta, \gp+ \delta]\times \bigl[
\lp_1',\lp_2'\bigr]$ we have
 $\lp \Lambda_\mu^\star\bigl(
\gp'/\lp\bigr) > \lp\Lambda_\mu^\star(\gp/
\lp) - \frac{\eps}{4}.$
\end{tabular}
\end{equation}
Using \eqref{eqpalpha_equals_0}, observe that if $\gp= 0$, then
$c_\mu(\gp)$ is less than $\gamma_\kappa(0,0)$, which implies that
$\lp_1>0$. Therefore, \eqref{eqnpunif_cont} follows in the case $\gp=
0$ from the lower semicontinuity of $\Lambda_\mu^\star$ at 0. For the
case $\gp> 0$, we observe that $\lp\Lambda_\mu^\star(\gp/\lp)$ finite
on $[\lp_1',\lp_2']$. By lower semicontinuity and convexity of
$\Lambda_\mu^\star$, this implies that $\lp\Lambda_\mu^\star(\gp
/\lp)$
is continuous on $[\lp_1',\lp_2']$. Since $[\lp_1',\lp_2']$ is compact,
we conclude that $\lp\Lambda_\mu^\star(\gp/\lp)$ is uniformly continuous
on $[\lp_1',\lp_2']$. Since $\lp\Lambda_\mu^\star(\gp'/\lp)$
can be
written as $\frac{\gp'}{\gp}\frac{\lp\gp}{\gp'}
\Lambda_\mu^\star(\gp/(\lp\gp/\gp'))$ (a straightforward limiting
argument shows that this equality holds even when $\lp= 0$), the
uniform continuity of $\Lambda_\mu^\star$ implies \eqref{eqnpunif_cont}
except possibly at the endpoints $\lp_1'$ and $\lp_2'$. However, since
$\Lambda_\mu^\star$ is lower semicontinuous, \eqref{eqnpunif_cont}~holds at $\lp_1'$ and $\lp_2'$ as well.

Recall the collection of balls $\CD^{r}$ for $r>0$ that we defined in
the proof of
Proposition~\ref{proppupper_bound}.
For $n\in\N$,
let
\[
\CQ^{n} \colonequals \bigl\{ Q \in\CD^{\exp(-n)} \dvtx
\TLoopsum _{z(\eps)}\bigl(e^{-n}\bigr) \in (\gp- \delta, \gp+
\delta) \bigr\}.
\]
Our goal is to show that for all $Q \in\CD^{\exp(-n)}$ and $n$ sufficiently
large,
%
%e5.15 #&#
\begin{equation}
\label{eqnpQinQ} \P\bigl[Q\in\CQ^{n}\bigr] \leq e^{-n (c_\mu(\gp) - \eps/2)}.
\end{equation}
The rest of the proof is similar to that of
Proposition~\ref{proppupper_bound}. To prove \eqref{eqnpQinQ}, we
abbreviate $\TLoopcount_{z(Q)}(e^{-n}) $ as $\TLoopcount$ and write
\[
\P\bigl[Q \in\CQ^{n}\bigr] = \E \bigl[ \P \bigl[
\TLoopsum_{z(Q)}\bigl(e^{-n}\bigr) \in (\gp- \delta, \gp+
\delta) | \TLoopcount \bigr] \bigr].
\]
We split the conditional probability according to the value of
$\TLoopcount$:
\begin{eqnarray*}
&&\P\bigl[Q \in\CQ^{n}\bigr]\\
&&\qquad\leq\E \bigl[ \mathbf{1}_{\{\TLoopcount\notin
[\lp_1,\lp_2]\}}
\P \bigl[ \TLoopsum_{z(Q)}\bigl(e^{-n}\bigr) \in(\gp- \delta,
\gp+ \delta) | \TLoopcount \bigr] \bigr]
\\
&&\qquad\quad{}+\E \bigl[ \mathbf{1}_{\{\TLoopcount\in[\lp_1,\lp_2]\setminus
[\lp_1',\lp_2']\}} \P \bigl[ \TLoopsum_{z(Q)}
\bigl(e^{-n}\bigr) \in(\gp- \delta, \gp+ \delta) | \TLoopcount \bigr]
\bigr]
\\
&&\qquad\quad{}+\E \bigl[ \mathbf{1}_{\{\TLoopcount\in[\lp_1',\lp_2']\}} \P \bigl[ \TLoopsum_{z(Q)}
\bigl(e^{-n}\bigr) \in(\gp- \delta, \gp+ \delta) | \TLoopcount \bigr]
\bigr].
\end{eqnarray*}
The first term on the right-hand side is bounded above by
$e^{-n(c_\mu(\gp)+o(1))}$, because of our choice of $\lp_1$ and
$\lp_2$. Similarly, the second term is bounded above by
$e^{-n(c_\mu(\gp)+o(1))}$ by \hyperref[thmpcramer]{Cram\'er's theorem}
and our choice of $\lp_1'$ and $\lp_2'$.
Thus it remains to show that the third term is bounded above by
$e^{-n(c_\mu(\gp)-\eps/2)}$ for all $n$ sufficiently large.
Multiplying and dividing by
$e^{-n\TLoopcount\Lambda_\kappa^\star(1/\TLoopcount)}$, applying
\hyperref[thmpcramer]{Cram\'er's theorem},
and using \eqref{eqnpunif_cont}, we find that for large enough $n$, the
third term is bounded above by
\begin{eqnarray*}
& &\E\bigl[ \mathbf{1}_{\{\TLoopcount\in[\lp_1',\lp_2]\}} e^{-n(\TLoopcount\Lambda_\mu^\star(\gp/ \TLoopcount) +
\TLoopcount
\Lambda_\kappa^\star(1/\TLoopcount)- \eps/2)} e^{n \TLoopcount
\Lambda_\kappa^\star(1/\TLoopcount)} \bigr]
\\
&&\qquad \leq e^{-n(c_\mu(\gp)- \eps/4)} \E \bigl[ \one_{\TLoopcount\in[\lp_1',\lp_2']}e^{n \TLoopcount
\Lambda_\kappa^\star(1/\TLoopcount)} \bigr].
\end{eqnarray*}

Thus, it remains to show that $\E [ \one_{\TLoopcount\in[\lp
_1',\lp_2']} e^{n
\TLoopcount\Lambda_\kappa^\star(1/\TLoopcount)}  ] =
e^{o(n)}$. Applying \eqref{eqnpcramer_strong} from
\hyperref[thmpcramer]{Cram\'er's theorem}, we find that if $\lp
_{\mathrm{typical}}
\leq
\lp_2'$, we have
\begin{eqnarray*}
&&\E \bigl[ \one_{\TLoopcount\in[\lp_{\mathrm{typical}},\lp_2']} e^{n \TLoopcount
\Lambda_\kappa^\star(1/\TLoopcount)} \bigr] \\
&&\qquad\leq \sum
_{k=1}^{\lceil(\lp_2'-\lp_{\mathrm{typical}})n\rceil} \E \bigl[\one_{\{
(k-1)/n \leq
\TLoopcount- \lp_{\mathrm{typical}}< k/n\}}e^{n
\TLoopcount\Lambda_\kappa^\star(1/\TLoopcount)}
\bigr]
\\
&&\qquad\leq\sum_{k=1}^{\lceil(\lp_2'-\lp_{\mathrm{typical}})n\rceil} e^{n f_\kappa(k/n)} \E
[\one_{\{(k-1)/n \leq
\TLoopcount- \lp_{\mathrm{typical}}\}} ]
\\
&&\qquad\leq\sum_{k=1}^{\lceil(\lp_2'-\lp_{\mathrm{typical}})n\rceil} e^{n [f_\kappa
(k/n) -
f_\kappa((k-1)/n)]},
\end{eqnarray*}
where $f_\kappa(\lp)\colonequals\lp
\Lambda_\kappa^\star (\frac{1}{\lp} )$. The mean value theorem
implies that the summand is bounded above by
$\exp(\sup_{\lp\in[\lp_{\mathrm{typical}},\lp_2']}f_\kappa'(\lp
))$. Therefore, the
expectation on the event $\{\TLoopcount\in[\lp_{\mathrm
{typical}},\lp_2']\}$ is
$O(n)$. An analogous argument gives the same bound for the expectation on
$\{\TLoopcount\in[\lp_1',\lp_{\mathrm{typical}}]\}$, so
\[
\E \bigl[ \one_{\TLoopcount\in[\lp_1',\lp_2']} e^{n \TLoopcount
\Lambda_\kappa^\star(1/\TLoopcount)} \bigr] = O(n) =
e^{o(n)}.
\]

Now consider the case $\Lambda_\mu^\star(0)=\infty$, which implies that
$\lp_0(0) = 0$. As in the case $\Lambda_\mu^\star(0)<\infty$, it suffices
to show that for every $\eps>0$, there exists $\delta>0$ such that
%
%e5.16 #&#
\begin{equation}
\label{eqpalpha_nu_0} \P \bigl[\bigl|\TLoopsum_z\bigl(e^{-n}\bigr)\bigr|
\leq\delta \bigr] \leq e^{-n(\gamma_\kappa(0,0) - \eps)}.
\end{equation}
Choose $\eta>0$ small enough that $\lp\Lambda_\kappa^\star(\lp)
\geq
1-2/\kappa-3\kappa/32-\eps/2$ whenever $\lp\in(0,\eta)$. Then
choose $\delta> 0$ small enough that $\Lambda_\mu^\star(x)\geq
2/\eta$ for all $x\in(-\delta/\eta,\delta/\eta)$ (this is
possible by lower
semicontinuity of $\Lambda^\star$). Again abbreviating
$\TLoopcount_{z}(e^{-n})$ as $\TLoopcount$, we write
\begin{eqnarray*}
\P \bigl[\bigl|\TLoopsum_z\bigl(e^{-n}\bigr)\bigr|\leq\delta \bigr]
&= &\E \bigl[\P \bigl[\bigl|\TLoopsum_z\bigl(e^{-n}\bigr)\bigr|\leq
\delta\given \TLoopcount \bigr] \bigr]
\\
&=&\E \bigl[\P \bigl[\bigl|\TLoopsum_z\bigl(e^{-n}\bigr)\bigr|\leq
\delta\given \TLoopcount \bigr]\one_{\TLoopcount\in[0,\eta)} \bigr]
\\
&&{}+\E \bigl[\P \bigl[\bigl|\TLoopsum_z\bigl(e^{-n}\bigr)\bigr|\leq
\delta \given\TLoopcount \bigr]\one_{\TLoopcount\in[\eta,\infty)} \bigr].
\end{eqnarray*}
Bounding the conditional probability by 1 and using our choice of $\eta$,
we see that the first term is bounded above by
$e^{-n(\gamma_\kappa(0,0) - \eps)}$. For the second
term, we note by~\eqref{eqnpcramer_strong} in
\hyperref[thmpcramer]{Cram\'er's theorem} that the conditional
probability is bounded above by
\[
2\exp \Bigl( -n \TLoopcount\inf_{|y|\leq\delta/ \TLoopcount
}\Lambda^\star(y)
\Bigr).
\]
On the event where $\TLoopcount$ is at least $\eta$, the factor
$\TLoopcount\inf_{|y|\leq\delta/ \TLoopcount}\Lambda^\star(y)$
is at least $2$, which
implies that the second term is bounded by $e^{-2n}$. This establishes
\eqref{eqpalpha_nu_0} and completes the proof.
\end{pf}

\begin{pf*}{Proof of Theorem~\ref{thmpgff_maximum}}
The logarithmic moment generating function of the signed Bernoulli
distribution is
$\Lambda_B(\mgfparamtwo) = \log\cosh(\gffparam\mgfparamtwo)$.
When $\kappa= 4$, formula~\eqref{eqnpLambdaR} for $\Lambda_\kappa$ simplifies to
\[
\Lambda_4(\mgfparam)= \cases{ -\log\cosh(\pi\sqrt{-2\mgfparam}), &\quad
$\mgfparam<0,$ \vspace *{2pt}
\cr
-\log\cos(\pi\sqrt{2 \mgfparam}), &\quad $\mgfparam
\geq0.$}
\]
Using the definition of the Fenchel--Legendre transform,
\[
\lp_0(\gp) \Lambda_B^\star \biggl(
\frac{\gp}{\lp_0(\gp)} \biggr) + \lp_0(\gp) \Lambda_4^\star
\biggl(\frac{1}{\lp_0(\gp)} \biggr)= \inf_{\lp\geq0} \sup
_{\mgfparamtwo,\mgfparam} \bigl[ \mgfparamtwo\gp+ \mgfparam- \lp\bigl(
\Lambda_4(\mgfparam) + \Lambda_B(\mgfparamtwo)\bigr)
\bigr].
\]
By the minimax theorem (see, e.g., \cite{pollard:minimax}), the
right-hand side equals
\[
\sup_{\mgfparamtwo,\mgfparam} \inf_{\lp\geq0} \bigl[ \mgfparamtwo
\gp+ \mgfparam- \lp\bigl(\Lambda_4(\mgfparam) +
\Lambda_B(\mgfparamtwo)\bigr) \bigr] = \sup_{\mgfparamtwo,\mgfparam\dvtx  \Lambda_4(\mgfparam) + \Lambda
_B(\mgfparamtwo) \leq0}
[\mgfparamtwo\gp+ \mgfparam ].
\]
Since $\Lambda_4(\mgfparam)$ is continuous in $\lambda$ and $\Lambda
_4(\mgfparam) \to
\infty$ as $\mgfparam\to\infty$, if $\Lambda_4(\mgfparam) +
\Lambda_B(\mgfparamtwo) < 0$, then $\mgfparam$ can be increased so that
$\Lambda_4(\mgfparam) + \Lambda_B(\mgfparamtwo) = 0$.
Thus, this last supremum can be replaced by the supremum over
$\mgfparam$ and $\mgfparamtwo$ satisfying $\Lambda_4(\mgfparam) +
\Lambda_B(\mgfparamtwo) = 0$.

Observe that $\Lambda_B(\mgfparamtwo) \geq0$ for all $\mgfparamtwo
$, and
$\Lambda_4(\mgfparam) < 0$ only when $\lambda< 0$. It follows that if
$\Lambda_4(\mgfparam) + \Lambda_B(\mgfparamtwo)= 0$, then $\lambda<
0$ and we can use the formulas for $\Lambda_4$ and $\Lambda_B$ to
conclude that
%
%e5.17 #&#
\begin{equation}
\label{eqnpmgf_zero_imp} \Lambda_4(\mgfparam) + \Lambda_B(
\mgfparamtwo)= 0 \qquad\mbox{implies } \gffparam\mgfparamtwo= \pi\sqrt{-2\mgfparam}.
\end{equation}
So we have
\begin{eqnarray*}
\lp_0(\gp) \Lambda_B^\star \biggl(
\frac{\gp}{\lp_0(\gp)} \biggr) + \lp_0(\gp) \Lambda_4^\star
\biggl(\frac{1}{\lp_0(\gp)} \biggr)&=& \sup_{\mgfparamtwo,\mgfparam\dvtx  \Lambda_4(\mgfparam) +
\Lambda_B(\mgfparamtwo) = 0} (\mgfparamtwo
\gp+ \mgfparam)
\\
&=& \sup_{\mgfparam< 0} \biggl(\frac{\gp\pi}{\gffparam} \sqrt {-2\mgfparam} +
\mgfparam \biggr)
\\
&=& \frac{\pi^2\gp^2}{2\gffparam^2} ,
\end{eqnarray*}
since the supremum is achieved when $\mgfparam= -\gp^2\pi
^2/2\gffparam^2$.
\end{pf*}

\begin{pf*}{Proof of Theorem~\ref{thmpgff_maximum_support}}
In light of Theorems \ref{thmparbitrary_distribution_joint} and
\ref{thmparbitrary_distribution}, it suffices to show that the maximum
of $2-\gamma_\kappa(\gp,\lp)$
is obtained when $\lp= \frac{\gp}{\gffparam} \coth(\frac{\pi^2
\gp}{\gffparam})$. As in the proof of
Theorem~\ref{thmpgff_maximum}, we begin by writing
\[
\gamma_\kappa(\gp,\lp) = \lp\Lambda_\kappa^\star
\biggl(\frac
{\gp}{\lp
} \biggr) + \lp\Lambda_\mu^\star
\biggl(\frac{1}{\lp} \biggr) = \sup_{\mgfparamtwo,\mgfparam} \bigl[
\mgfparamtwo\gp+ \mgfparam- \lp\bigl(\Lambda_\kappa(\mgfparam) +
\Lambda_\mu(\mgfparamtwo)\bigr) \bigr].
\]
At the minimizing value of $\lp$ and the corresponding maximizing values
of $\mgfparamtwo$ and~$\mgfparam$, the derivatives of the expression
in brackets with
respect to $\lp$, $\mgfparam$, and $\mgfparamtwo$ are all zero.
Differentiating,
we obtain the system
\begin{eqnarray*}
\Lambda_\kappa(\mgfparam) + \Lambda_\mu(\mgfparamtwo) &=& 0,
\\
\Lambda_\kappa'(\mgfparam) = \frac{1}{\gp}
\Lambda_\mu '(\mgfparamtwo) &=& \frac{1}{\lp}.
\end{eqnarray*}
The first equation implies $\gffparam\mgfparamtwo= \pi\sqrt
{-2\mgfparam}$ as in \eqref{eqnpmgf_zero_imp}. Substituting for
$\mgfparam$ in the equation $\Lambda_\kappa'(\mgfparam)
= \frac{1}{\gp} \Lambda_\mu'(\mgfparamtwo)$, we get $\gp=
\gffparam^2
\mgfparamtwo/\pi^2$. Finally, substituting into $\frac{1}{\gp}
\Lambda_\mu'(\mgfparamtwo)=1/\lp$ gives $\lp= \frac{\gp
}{\gffparam}
\coth (\frac{\pi^2 \gp}{\gffparam} )$, as desired.
\end{pf*}

%s6 #&#
\section{Further questions}
\label{secpquestions}

%\makeatletter{}\newcounter{question}
%\setcounter{question}{1}

One of the consequences of Theorem~\ref{thmpmain} is that for each
$\kappa\in(8/3,8)$ there exists a constant $c$ such that the
following is true. Almost surely,
\[
\sup_{z \in D} \SLoopcount_z(\eps) = c\bigl(1+o(1)
\bigr) \log(1/\eps)\qquad \mbox {as } \eps\to0.
\]
Is it possible to remove the $o(1)$ and give the order of the
correction term? In particular, in analogy with the work of Bramson and
Zeitouni \cite{BZ_TIGHT} for the discrete GFF, is it true that there
exist a constant $b$ such that
\[
\sup_{z \in D} \SLoopcount_z(\eps) = c \log(1/\eps) +
b \log\log (1/\eps)+O(1) \qquad\mbox{as } \eps\to0?
\]

Do the discrete loop models that are known to converge to CLE have the
same extreme nesting behavior as CLE?

%\appendix
\section*{Notation}
\label{secpnotation}
\begin{itemize}
\item$D$ is a simply connected proper domain in $\C$, that is,
$\varnothing\subsetneq D\subsetneq\C$ (page \pageref{notpdomD}).
\item$\Gamma$ denotes a $\CLE_\kappa$ process on $D$
(page \pageref{notpCLE}).
\item$\Loopcount_z(\eps)$ is the number of loops of $\Gamma$ which
surround $B(z,\eps)$ (page \pageref{notpLoopcount}).
\item$\Phi_\lp(\Gamma)$ is the set of all $z\in D$ such that
$\Loopcount_z(\eps) = (\lp+o(1))\log(1/\eps)$ as $\eps\to0$
[\eqref{notpPhi_lp} on page \pageref{notpPhi_lp}].
\item$\Loop_z$ is the sequence of loops of $\Gamma$ which surround $z$
(page \pageref{notploop}).
\item$\Loop_z^j$ is the $j$th loop of $\Gamma$ which surrounds $z$
(page \pageref{notploop}).
\item$U_z^j$ is the connected component of $D \setminus\Loop_z^j$ which
contains $z$ (page \pageref{notploop}).
\item$\gamma_\kappa(\lp)$ is the exponent for how unlikely it is
for a point
to be surrounded by a $\lp$ density of loops [\eqref{eqnpg(nu)} on
page \pageref{eqnpg(nu)}]:
\[
\log\P \bigl[J_{z,r}^\subset= \bigl(\lp+o(1)\bigr) \log(1/r)
\bigr] = \bigl(\gamma _\kappa(\lp )+o(1)\bigr)\log(1/r).
\]
\item$\Lambda_\kappa$ is the $\log$ moment
generating function for the $\log$ conformal radius distribution, and
$\Lambda_\kappa^*$ is its
Fenchel--Legendre transform (pages \pageref{notpfenchel}, \pageref
{notpfenchel_two}).
\item$\SLoopcount_z(\eps)$ is the sum of the loop weights over the loops
of $\Gamma$ which surround $B(z,\eps)$
[\eqref{eqnpnormalized_loop_count} on page \pageref{eqnpnormalized_loop_count}].
\item$\TLoopcount_z(\eps)$ and $\TLoopsum_z(\eps)$ are normalized versions
of $\Loopcount_z(\eps)$ and $\SLoopcount_z(\eps)$, obtained by dividing
by $\log(1/\eps)$ [\eqref{eqpTloopcount} on page \pageref{eqpTloopcount}
and \eqref{eqnpnormalized_loop_count} on page \pageref{eqnpnormalized_loop_count}].
\item$\mu$ is the weight distribution on the loops (page~\pageref{notpmu}).
\item$(T_i)_{i=1}^{\infty}$ is the sequence of $\log$ conformal radii
increments for $\CLE$ loops which surround a given point
(page \pageref{notpTi}).
\item$t_k = 2^{-(k+1)}$ and $s_k = \prod_{1\leq j < k}t_j$ (page~\pageref{notptk}). %and (page \pageref{notpsk})
\item$J_{z,r}^\cap$ is the index of the first loop of $\Loop_z$ which
intersects $B(z,r)$ [\eqref{notpJcapsubset} on page~\pageref
{notpJcapsubset}].
\item$J_{z,r}^\subset$ is the index of the first loop of $\Loop_z$ which
is contained in $B(z,r)$ [\eqref{notpJcapsubset} on page~\pageref{notpJcapsubset}].
\item$\Gamma_z(\eps)$ is the set of loops of $\Gamma$ which surround
$B(z,\eps)$ (page \pageref{notpGamma_z}).
\end{itemize}

\section*{Acknowledgements}
We thank Amir Dembo for providing us with the computation which
extracts Theorems \ref{thmpgff_maximum} and
\ref{thmpgff_maximum_support} from
Theorem~\ref{thmparbitrary_distribution}. We also thank Nike Sun for
helpful comments.
Both J. Miller and S.~S. Watson thank the hospitality of the Theory Group at Microsoft
Research, where part of the research for this work was completed.

%\makeatletter
%\def\@rst #1 #2other{#1}
%\renewcommand\MR[1]{\relax\ifhmode\unskip\spacefactor3000 \space\fi
% \MRhref{\expandafter\@rst #1 other}{#1}}
%\newcommand{\MRhref}[2]{
%\href{http://www.ams.org/mathscinet-getitem?mr=#1}{MR#1}}
%\makeatother
%
%\bibliographystyle{hmralphaabbrv}
%%\bibliography{../maxloops}
% imsref loaded by akundreckaite, 2015-02-23 10:05:10
%

%\begin{appendix}
%\section{}
%\end{appendix}

% zodis "Acknowledgments" paliekamas pagal autoriu
%\section*{Acknowledgments}

%\begin{supplement}[id=suppA]
%\sname{Supplement A}
%\stitle{}
%\slink[doi]{10.1214/00-AOPXXXXSUPP} %[doi,text={...}] - jei reikia
%suskaldyti doi
%\sdatatype{.pdf}
%\sfilename{aopXXXX\_supp.pdf}
%\sdescription{}
%\end{supplement}

%\begin{thebibliography}{99}
%\bibitem[\protect\citeauthoryear{}{}]{r1}
%\bibitem{r1}
%\end{thebibliography}

\printaddresses
\end{document}